\documentclass[12pt]{article}
\usepackage{mycommands}

\newtheorem{prop}{Proposition}
\newtheorem{theorem}[prop]{Theorem}
\newtheorem{lemma}[prop]{Lemma}

\theoremstyle{definition}
\newtheorem{definition}[prop]{Definition}
\theoremstyle{remark}

\title{On the Order of P-Strict Promotion on $V\times [\ell]$}
\author{Ben Adenbaum\thanks{Department of Mathematics, Dartmouth College, Hanover, NH 03755, USA.}}
\date{}

\begin{document}
\maketitle
\begin{abstract}
Denote by $\V$ the poset consisting of the elements $\{A,B,C\}$ with cover relations $\{A\lessdot B, A\lessdot C\}$. We show that $P$-strict promotion, as defined by Bernstein, Striker, and Vorland, on $P$-strict labelings of $\V\times [\ell]$ with labels in the set $[q]$ has order $2q$ for every $\ell\ge 1$ and $q\ge 3$. The idea of proof is to suitably generalize the methods of Hopkins and Rubey which they used to prove that Promotion on the linear extensions of $\V\times [n]$ has order $6n$. As a consequence of results of Bernstein, Striker, and Vorland, this result proves that piecewise-linear promotion, and thus picewise-linear rowmotion, on $\V\times [k]$ has order $2(k+2)$ for all $k\ge 1$, as conjectured by Hopkins. 
\end{abstract}
\section{Introduction}
 Promotion is an action on the linear extensions of a finite poset, see section~\ref{sec:PromotionBackground} for definitions.  Throughout we will be concerned with a generalization of promotion, due to Bernstein, Striker, and Vorland, named \emph{$P$-strict promotion}; see Section~\ref{sec:PromotionBackground} for the definition.

Similarly, rowmotion is an operation defined on the order ideals of a finite poset; see 
Section~\ref{sec:rowmotion} for definitions. 
Historically, rowmotion was first described by Brouwer and Schrijver~\cite{brouwer1974period}, and then again by Cameron and Fon-der-Flaass~\cite{cameron1995orbits} as a composition of certain involutions called toggles. The name of rowmotion comes from the work of Striker and Williams~\cite{striker2012promotion} where, for certain posets, rowmotion is described as a composition of toggles along the rows, and it is shown to be related to promotion. In particular, the toggle definition was extended to an action referred to as piecewise-linear rowmotion on the order polytope of $P$ in~\cite{einstein2021combinatorial}.

The family of posets of the form $\V\times [k]$ have been the focus of study, because of their ``good" dynamical behavior, first conjectured by Hopkins in~\cite{hopkins2020order}, especially with respect to both Sch\"{u}tzenberger promotion of linear extensions and rowmotion. In particular it has been shown that the orders of these actions for a fixed $k$ are $6k$~\cite{hopkins2022promotion} and $2(k+2)$~\cite{plante2022whirling} and~\cite{plante2024rowmotion} respectively. Furthermore it has been conjectured~\cite{hopkins2020order} that the order of piecewise-linear rowmotion on the order polytope of $\V\times [k]$ has the same order as rowmotion on the order ideals of $\V\times [k]$. The first goal of this work is to resolve this conjecture by proving an equivalent conjecture of Bernstein, Striker, and Vorland concerning the order of $P$-strict promotion on $P$-strict labelings of $\V\times[\ell]$ with entries in $[q]$ for all $\ell\ge 1, q\ge 3$. The formal statement is given in Theorem~\ref{thm:Main}. These two conjectures are equivalent as the action of $P$-strict promotion on $\V\times[\ell]$ with entries bounded by $q$ for all $\ell\ge 1$ for a fixed $q\ge 3$ is in equivariant bijection with piecewise-linear rowmotion on $\ell$-bounded $P$-partitions of $\V\times [q-2]$. By rescaling these are the rational points of the order polytope of $\V\times [q-2]$ whose denominators are divisible by $\ell$. The method of proof was suggested as a possible attack for this problem by Bernstein, Striker, and Vorland in~\cite{bernstein2024p}.

The paper is structured as follows. In Section~\ref{sec:Background} we review the necessary background for our proof. Section~\ref{sec:Proof} is devoted to the proof. We note that while the author was not aware of this at the time of discovery this argument, a similar idea of considering a modification of the arc diagrams of Hopkins and Rubey and studying it from the perspective of $P$-strict promotion was considered by Bernstein in~\cite{bernstein2022new}. In~\cite{bernstein2022new} a proof of Theorem~\ref{thm:Main} was not given.

 \subsection*{Acknowledgements} The author thanks Sergi Elizalde, Sam Hopkins, Tom Roby, Jessica Striker, and Alexander Wilson for helpful discussions.

\section{Background}\label{sec:Background}
In this section, we review the necessary background for the argument of Theorem~\ref{thm:Main}. Throughout $P$ will denote a poset. Recall that $P$ is said to be \emph{graded of rank} $n$ if every maximal chain of $P$ has $n+1$ elements. Denote by $rk$ the rank function of such a poset.  
\subsection{Promotion}\label{sec:PromotionBackground}
For $P$ a finite poset with $|P|=m$, recall a \emph{linear extension} of $P$ is an order-preserving bijection $f:P\to [m]$. Typically this is represented as a labeling of the Hasse diagram of $P$ with the elements of $[m]$ where if $p<_{P} p' $ then the label of $p'$ is greater than the label of $p$.  Equivalently $f$ is as an ordered tuple listing the elements of $P$ where if $p<_{P} p'$ then $p$ precedes $p'$ in the tuple. Denote by $e(P)$ the set of all linear extensions of $P$. For $1\le i \le m-1$, define the $i$th Bender--Knuth involution $t_i:e(P)\to e(P)$ by setting $t_i(f)$ to be the linear extension of $P$ obtained from $f\in e(P)$ by switching the labels $i$ and $i+1$ if the elements labeled by $i$ and $i+1$ are incomparable and doing nothing otherwise. Note that $t_i$ is an involution as two consecutive applications just swaps the labels of $i$ and $i+1$ twice or does nothing twice. Define \emph{promotion} on $e(P)$ to be the map $\Pro = t_{m-1}t_{m-2}\cdots t_{2}t_1$. See Figure~\ref{fig:Pro} for an example of promotion on $\V\times [6]$.

\begin{figure}[h]
    \centering
    \begin{tikzpicture}
        \node (A1) at (0,0) {1};
        \node (A2) at (0,1) {3};
        \node (A3) at (0,2) {6};
        \node (A4) at (0,3) {7};
        \node (A5) at (0,4) {11};
        \node (A6) at (0,5) {14};
        \node (B1) at (-1,1) {4};
        \node (B2) at (-1,2) {5};
        \node (B3) at (-1,3) {8};
        \node (B4) at (-1,4) {13};
        \node (B5) at (-1,5) {15};
        \node (B6) at (-1,6) {18};
        \node (C1) at (1,1) {2};
        \node (C2) at (1,2) {9};
        \node (C3) at (1,3) {10};
        \node (C4) at (1,4) {12};
        \node (C5) at (1,5) {16};
        \node (C6) at (1,6) {17};
        \draw[-] (A1) -- (A2);
        \draw[-] (A2) -- (A3);
        \draw[-] (A3) -- (A4);
        \draw[-] (A4) -- (A5);
        \draw[-] (A5) -- (A6);
        \draw[-] (B1) -- (B2);
        \draw[-] (B2) -- (B3);
        \draw[-] (B3) -- (B4);
        \draw[-] (B4) -- (B5);
        \draw[-] (B5) -- (B6);
        \draw[-] (C1) -- (C2);
        \draw[-] (C2) -- (C3);
        \draw[-] (C3) -- (C4);
        \draw[-] (C4) -- (C5);
        \draw[-] (C5) -- (C6);
        \draw[-] (A1) -- (B1);
        \draw[-] (A1) -- (C1);
        \draw[-] (A2) -- (B2);
        \draw[-] (A2) -- (C2);
        \draw[-] (A3) -- (B3);
        \draw[-] (A3) -- (C3);
        \draw[-] (A4) -- (B4);
        \draw[-] (A4) -- (C4);
        \draw[-] (A5) -- (B5);
        \draw[-] (A5) -- (C5);
        \draw[-] (A6) -- (B6);
        \draw[-] (A6) -- (C6);
        \draw[-to] (1.5, 3) --node[above]{$\Pro$} (2,3);
        \begin{scope}[shift = {(3.5,0)}]
       \node (A1) at (0,0) {1};
        \node (A2) at (0,1) {2};
        \node (A3) at (0,2) {5};
        \node (A4) at (0,3) {6};
        \node (A5) at (0,4) {10};
        \node (A6) at (0,5) {13};
        \node (B1) at (-1,1) {3};
        \node (B2) at (-1,2) {4};
        \node (B3) at (-1,3) {7};
        \node (B4) at (-1,4) {12};
        \node (B5) at (-1,5) {14};
        \node (B6) at (-1,6) {17};
        \node (C1) at (1,1) {8};
        \node (C2) at (1,2) {9};
        \node (C3) at (1,3) {11};
        \node (C4) at (1,4) {15};
        \node (C5) at (1,5) {16};
        \node (C6) at (1,6) {18};
        \draw[-] (A1) -- (A2);
        \draw[-] (A2) -- (A3);
        \draw[-] (A3) -- (A4);
        \draw[-] (A4) -- (A5);
        \draw[-] (A5) -- (A6);
        \draw[-] (B1) -- (B2);
        \draw[-] (B2) -- (B3);
        \draw[-] (B3) -- (B4);
        \draw[-] (B4) -- (B5);
        \draw[-] (B5) -- (B6);
        \draw[-] (C1) -- (C2);
        \draw[-] (C2) -- (C3);
        \draw[-] (C3) -- (C4);
        \draw[-] (C4) -- (C5);
        \draw[-] (C5) -- (C6);
        \draw[-] (A1) -- (B1);
        \draw[-] (A1) -- (C1);
        \draw[-] (A2) -- (B2);
        \draw[-] (A2) -- (C2);
        \draw[-] (A3) -- (B3);
        \draw[-] (A3) -- (C3);
        \draw[-] (A4) -- (B4);
        \draw[-] (A4) -- (C4);
        \draw[-] (A5) -- (B5);
        \draw[-] (A5) -- (C5);
        \draw[-] (A6) -- (B6);
        \draw[-] (A6) -- (C6);
        \end{scope}
\end{tikzpicture}
\caption{An application of promotion applied to a linear extension of $\V\times [6].$}
    \label{fig:Pro}
\end{figure}
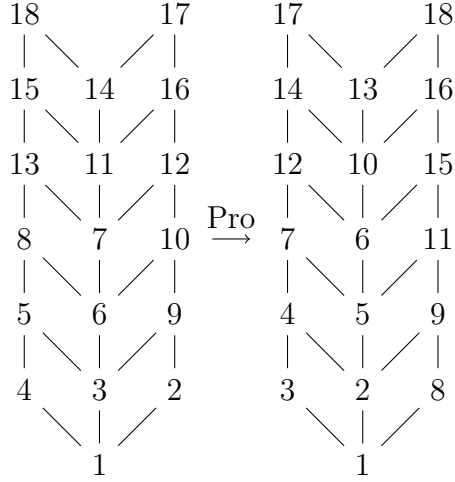
\subsection{P-Strict Promotion}\label{sec:PPromotionBackground}
We now state some preliminary definitions, following the treatment given in \cite{bernstein2024p}. The initial definitions and a more general treatment of these ideas can be found in \cite{bernstein2021p}. Notationally, using the convention of \cite{bernstein2024p}, $\mathcal{P}(\Z)$ is the set of finite subsets of $\Z$. For an example highlighting the next two definitions, see Figure~\ref{fig:V69Example}.
\begin{definition}[{\cite[Definition \ 2.2 and 2.3]{bernstein2024p}}]\label{def:PStrictLabelingAndConsistentRestriction}
A function $f:P\times [\ell]\to \Z$ is a $P$-\emph{strict labeling} of $P\times [\ell]$ \emph{with restriction function} $R:P\to \mathcal{P}(\Z)$ if $f$ satisfies the following:
\begin{enumerate}[(1)]
\item $f(p_1,i) < f(p_2,i)$ if $p_1 <_{P} p_2$, edges strictly increase along copies of $P$.
\item $f(p,i_1) \le f(p,i_2)$ if $i_2 \le i_2$, edges weakly increase along copies of $[\ell]$.
\item $f(p,i) \in R(p)$, the function takes on values given by the restriction function $R$.
\end{enumerate}

A restriction function $R$ is \emph{consistent} with respect to $P\times [\ell]$ if for all $p\in P$ and $k\in R(p)$ there exists some $P$-strict labeling $f$ of $P\times [\ell]$ with $f(p,i)=k$, $1\le i \le \ell$.
\end{definition}
Continuing to follow the notation of \cite{bernstein2024p}, for a fixed $i\in [\ell]$ we refer to $L_i=\{(p,i):p\in P\}$ as the $i$th \emph{layer} of $f$, and for $p\in P$ we call $F_p=\{(p,i)|i\in [\ell]\}$ the $p$th \emph{fiber} of $P\times [\ell]$. Additionally, we denote the set of $P$-strict labelings on $P\times [\ell]$ with restriction function $R$ by $\mathcal{L}_{P\times [\ell]}(R)$. If $R$ is the consistent restriction function induced by the respective lower and upper bounds $a,b:P\to \Z$, i.e. $R(p)$ is the largest subinterval of $[a(p),b(p)]$ that allows $R$ to be consistent, then we denote this restriction function by $R^{b}_{a}$. For our purposes we will only work in the case where $a=1,b=q$ and we denote this restriction function by $R^q$.

In the case where $R=R^q$ and $P$ is graded of rank $n$, then $R(p) = \{rk(p) +i| i \in [q-(n+1)]\}$ for all $p\in P$.
\begin{definition}[{\cite[Definition \ 2.5]{bernstein2024p}}]
Let $R(p)_{>k}$ denote the smallest label of $R(p)$ that is larger than $k$, and let $R(p)_{<k}$ denote the largest label of $R(p)$ less than $k$. If $R= R^q$, then $R(p)_{>k}$ and $R(p)_{<k}$ are $k+1$ and $k-1$ respectively if they exist.

Say that a label $f(p,i)$ in a $P$-strict labeling $f\in \mathcal{L}_{P \times [\ell]}(R)$ is \emph{raisable (lowerable)} if there exists another $P$-strict labeling  $g\in\mathcal{L}_{P \times [\ell]}(R)$ where $f(p,i)<g(p,i)$ ($f(p,i)>g(p,i)$), and $f(p',i')=g(p',i')$ for all $(p',i')\in P \times [\ell]$, with $p' \neq p$. 
\end{definition}

\begin{definition}[{\cite[Definition 2.6]{bernstein2024p}}]
\label{def:BenderKnuth}
Define the action of the \emph{$k$th Bender--Knuth involution} $\tau_k$ on 
a $P$-strict labeling $f\in \mathcal{L}_{P \times [\ell]}(R)$ be as follows: 
identify all raisable labels $f(p,i)=k$ and all lowerable labels $f(p,i)=R(p)_{>k}$. 
Call these labels `free'. Suppose the labels $f(F_p)$ include $a$ free $k$ labels followed by $b$ free $R(p)_{>k}$ labels; $\tau_k$ changes these labels to $b$ copies of $k$ followed by $a$ copies of $R(p)_{>k}$. 
\emph{Promotion} on $P$-strict labelings is defined as the composition of these involutions: $\Pro(f)=\cdots\circ\tau_3\circ\tau_2\circ\tau_1 \circ \cdots(f)$. Note that since $R$ induces upper and lower bounds on the labels, only a finite number of Bender--Knuth involutions act nontrivially.
\end{definition}
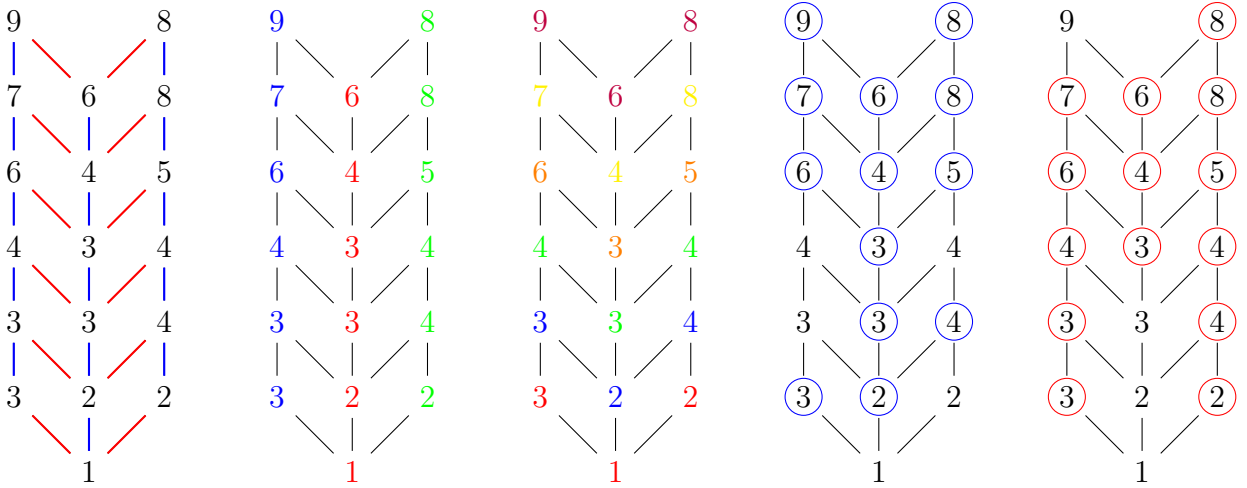
\begin{figure}[htb]
    \centering
    \begin{tikzpicture}
        \node (A1) at (0,0) {1};
        \node (A2) at (0,1) {2};
        \node (A3) at (0,2) {3};
        \node (A4) at (0,3) {3};
        \node (A5) at (0,4) {4};
        \node (A6) at (0,5) {6};
        \node (B1) at (-1,1) {3};
        \node (B2) at (-1,2) {3};
        \node (B3) at (-1,3) {4};
        \node (B4) at (-1,4) {6};
        \node (B5) at (-1,5) {7};
        \node (B6) at (-1,6) {9};
        \node (C1) at (1,1) {2};
        \node (C2) at (1,2) {4};
        \node (C3) at (1,3) {4};
        \node (C4) at (1,4) {5};
        \node (C5) at (1,5) {8};
        \node (C6) at (1,6) {8};
        \draw[-] (A1) -- (A2);
        \draw[-] (A2) -- (A3);
        \draw[-] (A3) -- (A4);
        \draw[-] (A4) -- (A5);
        \draw[-] (A5) -- (A6);
        \draw[-] (B1) -- (B2);
        \draw[-] (B2) -- (B3);
        \draw[-] (B3) -- (B4);
        \draw[-] (B4) -- (B5);
        \draw[-] (B5) -- (B6);
        \draw[-] (C1) -- (C2);
        \draw[-] (C2) -- (C3);
        \draw[-] (C3) -- (C4);
        \draw[-] (C4) -- (C5);
        \draw[-] (C5) -- (C6);
        \draw[-] (A1) -- (B1);
        \draw[-] (A1) -- (C1);
        \draw[-] (A2) -- (B2);
        \draw[-] (A2) -- (C2);
        \draw[-] (A3) -- (B3);
        \draw[-] (A3) -- (C3);
        \draw[-] (A4) -- (B4);
        \draw[-] (A4) -- (C4);
        \draw[-] (A5) -- (B5);
        \draw[-] (A5) -- (C5);
        \draw[-] (A6) -- (B6);
        \draw[-] (A6) -- (C6);
        \draw[-, thick, color = red] (A1) -- (B1);
        \draw[-, thick, color = red] (A1) -- (C1);
        \draw[-, thick, color = red] (A2) -- (B2);
        \draw[-, thick, color = red] (A2) -- (C2);
        \draw[-, thick, color = red] (A3) -- (B3);
        \draw[-, thick, color = red] (A3) -- (C3);
        \draw[-, thick, color = red] (A4) -- (B4);
        \draw[-, thick, color = red] (A4) -- (C4);
        \draw[-, thick, color = red] (A5) -- (B5);
        \draw[-, thick, color = red] (A5) -- (C5);
        \draw[-, thick, color = red] (A6) -- (B6);
        \draw[-, thick, color = red] (A6) -- (C6);
        \draw[-, thick, color = blue]  (A1) -- (A2);
        \draw[-, thick, color = blue]  (A2) -- (A3);
        \draw[-, thick, color = blue]  (A3) -- (A4);
        \draw[-, thick, color = blue]  (A4) -- (A5);
        \draw[-, thick, color = blue]  (A5) -- (A6);
        \draw[-, thick, color = blue]  (B1) -- (B2);
        \draw[-, thick, color = blue]  (B2) -- (B3);
        \draw[-, thick, color = blue]  (B3) -- (B4);
        \draw[-, thick, color = blue]  (B4) -- (B5);
        \draw[-, thick, color = blue]  (B5) -- (B6);
        \draw[-, thick, color = blue]  (C1) -- (C2);
        \draw[-, thick, color = blue]  (C2) -- (C3);
        \draw[-, thick, color = blue]  (C3) -- (C4);
        \draw[-, thick, color = blue]  (C4) -- (C5);
        \draw[-, thick, color = blue]  (C5) -- (C6);
        \begin{scope}[shift = {(3.5, 0)}]
        \node (A1) at (0,0) {\textcolor{red}{1}};
        \node (A2) at (0,1) {\textcolor{red}{2}};
        \node (A3) at (0,2) {\textcolor{red}{3}};
        \node (A4) at (0,3) {\textcolor{red}{3}};
        \node (A5) at (0,4) {\textcolor{red}{4}};
        \node (A6) at (0,5) {\textcolor{red}{6}};
        \node (B1) at (-1,1) {\textcolor{blue}{3}};
        \node (B2) at (-1,2) {\textcolor{blue}{3}};
        \node (B3) at (-1,3) {\textcolor{blue}{4}};
        \node (B4) at (-1,4) {\textcolor{blue}{6}};
        \node (B5) at (-1,5) {\textcolor{blue}{7}};
        \node (B6) at (-1,6) {\textcolor{blue}{9}};
        \node (C1) at (1,1) {\textcolor{green}{2}};
        \node (C2) at (1,2) {\textcolor{green}{4}};
        \node (C3) at (1,3) {\textcolor{green}{4}};
        \node (C4) at (1,4) {\textcolor{green}{5}};
        \node (C5) at (1,5) {\textcolor{green}{8}};
        \node (C6) at (1,6) {\textcolor{green}{8}};
        \draw[-] (A1) -- (A2);
        \draw[-] (A2) -- (A3);
        \draw[-] (A3) -- (A4);
        \draw[-] (A4) -- (A5);
        \draw[-] (A5) -- (A6);
        \draw[-] (B1) -- (B2);
        \draw[-] (B2) -- (B3);
        \draw[-] (B3) -- (B4);
        \draw[-] (B4) -- (B5);
        \draw[-] (B5) -- (B6);
        \draw[-] (C1) -- (C2);
        \draw[-] (C2) -- (C3);
        \draw[-] (C3) -- (C4);
        \draw[-] (C4) -- (C5);
        \draw[-] (C5) -- (C6);
        \draw[-] (A1) -- (B1);
        \draw[-] (A1) -- (C1);
        \draw[-] (A2) -- (B2);
        \draw[-] (A2) -- (C2);
        \draw[-] (A3) -- (B3);
        \draw[-] (A3) -- (C3);
        \draw[-] (A4) -- (B4);
        \draw[-] (A4) -- (C4);
        \draw[-] (A5) -- (B5);
        \draw[-] (A5) -- (C5);
        \draw[-] (A6) -- (B6);
        \draw[-] (A6) -- (C6);
        \end{scope}
        \begin{scope}[shift = {(7, 0)}]
        \node (A1) at (0,0) {\textcolor{red}{1}};
        \node (A2) at (0,1) {\textcolor{blue}{2}};
        \node (A3) at (0,2) {\textcolor{green}{3}};
        \node (A4) at (0,3) {\textcolor{orange}{3}};
        \node (A5) at (0,4) {\textcolor{yellow}{4}};
        \node (A6) at (0,5) {\textcolor{purple}{6}};
        \node (B1) at (-1,1) {\textcolor{red}{3}};
        \node (B2) at (-1,2) {\textcolor{blue}{3}};
        \node (B3) at (-1,3) {\textcolor{green}{4}};
        \node (B4) at (-1,4) {\textcolor{orange}{6}};
        \node (B5) at (-1,5) {\textcolor{yellow}{7}};
        \node (B6) at (-1,6) {\textcolor{purple}{9}};
        \node (C1) at (1,1) {\textcolor{red}{2}};
        \node (C2) at (1,2) {\textcolor{blue}{4}};
        \node (C3) at (1,3) {\textcolor{green}{4}};
        \node (C4) at (1,4) {\textcolor{orange}{5}};
        \node (C5) at (1,5) {\textcolor{yellow}{8}};
        \node (C6) at (1,6) {\textcolor{purple}{8}};
        \draw[-] (A1) -- (A2);
        \draw[-] (A2) -- (A3);
        \draw[-] (A3) -- (A4);
        \draw[-] (A4) -- (A5);
        \draw[-] (A5) -- (A6);
        \draw[-] (B1) -- (B2);
        \draw[-] (B2) -- (B3);
        \draw[-] (B3) -- (B4);
        \draw[-] (B4) -- (B5);
        \draw[-] (B5) -- (B6);
        \draw[-] (C1) -- (C2);
        \draw[-] (C2) -- (C3);
        \draw[-] (C3) -- (C4);
        \draw[-] (C4) -- (C5);
        \draw[-] (C5) -- (C6);
        \draw[-] (A1) -- (B1);
        \draw[-] (A1) -- (C1);
        \draw[-] (A2) -- (B2);
        \draw[-] (A2) -- (C2);
        \draw[-] (A3) -- (B3);
        \draw[-] (A3) -- (C3);
        \draw[-] (A4) -- (B4);
        \draw[-] (A4) -- (C4);
        \draw[-] (A5) -- (B5);
        \draw[-] (A5) -- (C5);
        \draw[-] (A6) -- (B6);
        \draw[-] (A6) -- (C6);
        \end{scope}
    
       \begin{scope}[shift = {(10.5, 0)}]
        \node (A1) at (0,0) {1};
        \node (A2) at (0,1) {2};
        \node (A3) at (0,2) {3};
        \node (A4) at (0,3) {3};
        \node (A5) at (0,4) {4};
        \node (A6) at (0,5) {6};
        \node (B1) at (-1,1) {3};
        \node (B2) at (-1,2) {3};
        \node (B3) at (-1,3) {4};
        \node (B4) at (-1,4) {6};
        \node (B5) at (-1,5) {7};
        \node (B6) at (-1,6) {9};
        \node (C1) at (1,1) {2};
        \node (C2) at (1,2) {4};
        \node (C3) at (1,3) {4};
        \node (C4) at (1,4) {5};
        \node (C5) at (1,5) {8};
        \node (C6) at (1,6) {8};
        \draw[color = blue] (A2) circle (7pt);
        \draw[color = blue] (A3) circle (7pt);
        \draw[color = blue] (A4) circle (7pt);
        \draw[color = blue] (A5) circle (7pt);
        \draw[color = blue] (A6) circle (7pt);
        \draw[color = blue] (B1) circle (7pt);
        \draw[color = blue] (B4) circle (7pt);
        \draw[color = blue] (B5) circle (7pt);
        \draw[color = blue] (B6) circle (7pt);
        \draw[color = blue] (C2) circle (7pt);
        \draw[color = blue] (C4) circle (7pt);
        \draw[color = blue] (C5) circle (7pt);
        \draw[color = blue] (C6) circle (7pt);
        \draw[-] (A1) -- (A2);
        \draw[-] (A2) -- (A3);
        \draw[-] (A3) -- (A4);
        \draw[-] (A4) -- (A5);
        \draw[-] (A5) -- (A6);
        \draw[-] (B1) -- (B2);
        \draw[-] (B2) -- (B3);
        \draw[-] (B3) -- (B4);
        \draw[-] (B4) -- (B5);
        \draw[-] (B5) -- (B6);
        \draw[-] (C1) -- (C2);
        \draw[-] (C2) -- (C3);
        \draw[-] (C3) -- (C4);
        \draw[-] (C4) -- (C5);
        \draw[-] (C5) -- (C6);
        \draw[-] (A1) -- (B1);
        \draw[-] (A1) -- (C1);
        \draw[-] (A2) -- (B2);
        \draw[-] (A2) -- (C2);
        \draw[-] (A3) -- (B3);
        \draw[-] (A3) -- (C3);
        \draw[-] (A4) -- (B4);
        \draw[-] (A4) -- (C4);
        \draw[-] (A5) -- (B5);
        \draw[-] (A5) -- (C5);
        \draw[-] (A6) -- (B6);
        \draw[-] (A6) -- (C6);
        \end{scope}
        \begin{scope}[shift = {(14, 0)}]
              \node (A1) at (0,0) {1};
        \node (A2) at (0,1) {2};
        \node (A3) at (0,2) {3};
        \node (A4) at (0,3) {3};
        \node (A5) at (0,4) {4};
        \node (A6) at (0,5) {6};
        \node (B1) at (-1,1) {3};
        \node (B2) at (-1,2) {3};
        \node (B3) at (-1,3) {4};
        \node (B4) at (-1,4) {6};
        \node (B5) at (-1,5) {7};
        \node (B6) at (-1,6) {9};
        \node (C1) at (1,1) {2};
        \node (C2) at (1,2) {4};
        \node (C3) at (1,3) {4};
        \node (C4) at (1,4) {5};
        \node (C5) at (1,5) {8};
        \node (C6) at (1,6) {8};
                \draw[-] (A1) -- (A2);
        \draw[-] (A2) -- (A3);
        \draw[-] (A3) -- (A4);
        \draw[-] (A4) -- (A5);
        \draw[-] (A5) -- (A6);
        \draw[-] (B1) -- (B2);
        \draw[-] (B2) -- (B3);
        \draw[-] (B3) -- (B4);
        \draw[-] (B4) -- (B5);
        \draw[-] (B5) -- (B6);
        \draw[-] (C1) -- (C2);
        \draw[-] (C2) -- (C3);
        \draw[-] (C3) -- (C4);
        \draw[-] (C4) -- (C5);
        \draw[-] (C5) -- (C6);
        \draw[-] (A1) -- (B1);
        \draw[-] (A1) -- (C1);
        \draw[-] (A2) -- (B2);
        \draw[-] (A2) -- (C2);
        \draw[-] (A3) -- (B3);
        \draw[-] (A3) -- (C3);
        \draw[-] (A4) -- (B4);
        \draw[-] (A4) -- (C4);
        \draw[-] (A5) -- (B5);
        \draw[-] (A5) -- (C5);
        \draw[-] (A6) -- (B6);
        \draw[-] (A6) -- (C6);
        \draw[color = red] (A4) circle (7pt);
        \draw[color = red] (A5) circle (7pt);
        \draw[color = red] (A6) circle (7pt);
        \draw[color = red] (B1) circle (7pt);
        \draw[color = red] (B2) circle (7pt);
        \draw[color = red] (B3) circle (7pt);
        \draw[color = red] (B4) circle (7pt);
        \draw[color = red] (B5) circle (7pt);
        \draw[color = red] (C1) circle (7pt);
        \draw[color = red] (C2) circle (7pt);
        \draw[color = red] (C3) circle (7pt);
        \draw[color = red] (C4) circle (7pt);
        \draw[color = red] (C5) circle (7pt);
        \draw[color = red] (C6) circle (7pt);
        \end{scope}
    \end{tikzpicture}
    \caption{An example $f$ in $\mathcal{L}_{\V\times [6]}(R^9)$}
    \label{fig:V69Example}
\end{figure}
To more clearly explain the example of Figure~\ref{fig:V69Example}, when reading the copies of $f$ from left to right, the first copy has the weak edges colored blue and the strict edges colored red. In the second copy $F_B$ is colored blue, $F_A$ is colored red, $F_C$ is colored green. In the third copy, each layer has a distinct color. In the fourth copy all of the lowerable labels are circled in blue. In the fifth copy all of the raisable labels are circled in red.

We now state our main result. 

\begin{theorem}\label{thm:Main}
The order of $\Pro$ on $\mathcal{L}_{\V\times [\ell]}(R^q)$ divides $2q$.
\end{theorem}

When restricting to where $q= \ell|P|$, and all the labels are distinct, $P$-strict promotion is the same as promotion on the linear extensions of $P\times [\ell]$.
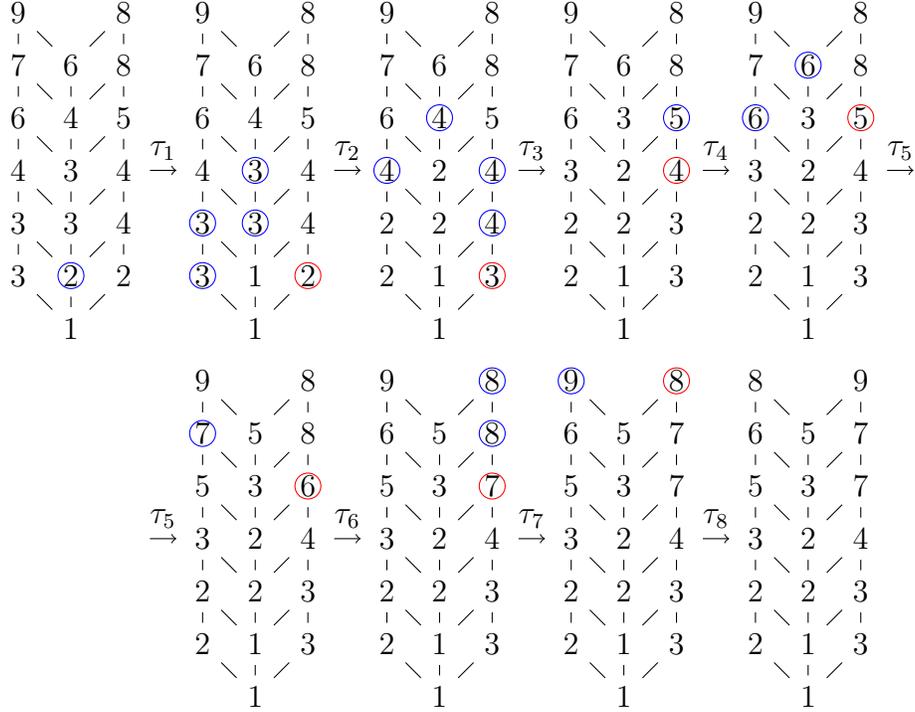
\begin{figure}[htb]
    \centering
    \begin{tikzpicture}[scale = .7]
        \node (A1) at (0,0) {1};
        \node (A2) at (0,1) {2};
        \node (A3) at (0,2) {3};
        \node (A4) at (0,3) {3};
        \node (A5) at (0,4) {4};
        \node (A6) at (0,5) {6};
        \node (B1) at (-1,1) {3};
        \node (B2) at (-1,2) {3};
        \node (B3) at (-1,3) {4};
        \node (B4) at (-1,4) {6};
        \node (B5) at (-1,5) {7};
        \node (B6) at (-1,6) {9};
        \node (C1) at (1,1) {2};
        \node (C2) at (1,2) {4};
        \node (C3) at (1,3) {4};
        \node (C4) at (1,4) {5};
        \node (C5) at (1,5) {8};
        \node (C6) at (1,6) {8};
        \draw[-] (A1) -- (A2);
        \draw[-] (A2) -- (A3);
        \draw[-] (A3) -- (A4);
        \draw[-] (A4) -- (A5);
        \draw[-] (A5) -- (A6);
        \draw[-] (B1) -- (B2);
        \draw[-] (B2) -- (B3);
        \draw[-] (B3) -- (B4);
        \draw[-] (B4) -- (B5);
        \draw[-] (B5) -- (B6);
        \draw[-] (C1) -- (C2);
        \draw[-] (C2) -- (C3);
        \draw[-] (C3) -- (C4);
        \draw[-] (C4) -- (C5);
        \draw[-] (C5) -- (C6);
        \draw[-] (A1) -- (B1);
        \draw[-] (A1) -- (C1);
        \draw[-] (A2) -- (B2);
        \draw[-] (A2) -- (C2);
        \draw[-] (A3) -- (B3);
        \draw[-] (A3) -- (C3);
        \draw[-] (A4) -- (B4);
        \draw[-] (A4) -- (C4);
        \draw[-] (A5) -- (B5);
        \draw[-] (A5) -- (C5);
        \draw[-] (A6) -- (B6);
        \draw[-] (A6) -- (C6);
        \draw[-to] (1.5, 3) --node[above]{$\tau_1$} (2,3);
        \draw[color = blue] (A2) circle (7pt);
        \begin{scope}[shift={(3.5,0)}]
        \node (A1) at (0,0) {1};
        \node (A2) at (0,1) {1};
        \node (A3) at (0,2) {3};
        \node (A4) at (0,3) {3};
        \node (A5) at (0,4) {4};
        \node (A6) at (0,5) {6};
        \node (B1) at (-1,1) {3};
        \node (B2) at (-1,2) {3};
        \node (B3) at (-1,3) {4};
        \node (B4) at (-1,4) {6};
        \node (B5) at (-1,5) {7};
        \node (B6) at (-1,6) {9};
        \node (C1) at (1,1) {2};
        \node (C2) at (1,2) {4};
        \node (C3) at (1,3) {4};
        \node (C4) at (1,4) {5};
        \node (C5) at (1,5) {8};
        \node (C6) at (1,6) {8};
        \draw[-] (A1) -- (A2);
        \draw[-] (A2) -- (A3);
        \draw[-] (A3) -- (A4);
        \draw[-] (A4) -- (A5);
        \draw[-] (A5) -- (A6);
        \draw[-] (B1) -- (B2);
        \draw[-] (B2) -- (B3);
        \draw[-] (B3) -- (B4);
        \draw[-] (B4) -- (B5);
        \draw[-] (B5) -- (B6);
        \draw[-] (C1) -- (C2);
        \draw[-] (C2) -- (C3);
        \draw[-] (C3) -- (C4);
        \draw[-] (C4) -- (C5);
        \draw[-] (C5) -- (C6);
        \draw[-] (A1) -- (B1);
        \draw[-] (A1) -- (C1);
        \draw[-] (A2) -- (B2);
        \draw[-] (A2) -- (C2);
        \draw[-] (A3) -- (B3);
        \draw[-] (A3) -- (C3);
        \draw[-] (A4) -- (B4);
        \draw[-] (A4) -- (C4);
        \draw[-] (A5) -- (B5);
        \draw[-] (A5) -- (C5);
        \draw[-] (A6) -- (B6);
        \draw[-] (A6) -- (C6);
        \draw[-to] (1.5, 3) --node[above]{$\tau_2$} (2,3);
        \draw[color = blue] (B1) circle (7pt);
        \draw[color = blue] (B2) circle (7pt);
        \draw[color = blue] (A3) circle (7pt);
        \draw[color = blue] (A4) circle (7pt);
        \draw[color = red] (C1) circle (7pt);
        \end{scope}
        \begin{scope}[shift={(7,0)}]
        \node (A1) at (0,0) {1};
        \node (A2) at (0,1) {1};
        \node (A3) at (0,2) {2};
        \node (A4) at (0,3) {2};
        \node (A5) at (0,4) {4};
        \node (A6) at (0,5) {6};
        \node (B1) at (-1,1) {2};
        \node (B2) at (-1,2) {2};
        \node (B3) at (-1,3) {4};
        \node (B4) at (-1,4) {6};
        \node (B5) at (-1,5) {7};
        \node (B6) at (-1,6) {9};
        \node (C1) at (1,1) {3};
        \node (C2) at (1,2) {4};
        \node (C3) at (1,3) {4};
        \node (C4) at (1,4) {5};
        \node (C5) at (1,5) {8};
        \node (C6) at (1,6) {8};
        \draw[-] (A1) -- (A2);
        \draw[-] (A2) -- (A3);
        \draw[-] (A3) -- (A4);
        \draw[-] (A4) -- (A5);
        \draw[-] (A5) -- (A6);
        \draw[-] (B1) -- (B2);
        \draw[-] (B2) -- (B3);
        \draw[-] (B3) -- (B4);
        \draw[-] (B4) -- (B5);
        \draw[-] (B5) -- (B6);
        \draw[-] (C1) -- (C2);
        \draw[-] (C2) -- (C3);
        \draw[-] (C3) -- (C4);
        \draw[-] (C4) -- (C5);
        \draw[-] (C5) -- (C6);
        \draw[-] (A1) -- (B1);
        \draw[-] (A1) -- (C1);
        \draw[-] (A2) -- (B2);
        \draw[-] (A2) -- (C2);
        \draw[-] (A3) -- (B3);
        \draw[-] (A3) -- (C3);
        \draw[-] (A4) -- (B4);
        \draw[-] (A4) -- (C4);
        \draw[-] (A5) -- (B5);
        \draw[-] (A5) -- (C5);
        \draw[-] (A6) -- (B6);
        \draw[-] (A6) -- (C6);
        \draw[-to] (1.5, 3) --node[above]{$\tau_3$} (2,3);
        \draw[color = blue] (B3) circle (7pt);
        \draw[color = blue] (A5) circle (7pt);
        \draw[color = blue] (C2) circle (7pt);
        \draw[color = blue] (C3) circle (7pt);
        \draw[color = red] (C1) circle (7pt);
        \end{scope}
        \begin{scope}[shift={(10.5,0)}]
        \node (A1) at (0,0) {1};
        \node (A2) at (0,1) {1};
        \node (A3) at (0,2) {2};
        \node (A4) at (0,3) {2};
        \node (A5) at (0,4) {3};
        \node (A6) at (0,5) {6};
        \node (B1) at (-1,1) {2};
        \node (B2) at (-1,2) {2};
        \node (B3) at (-1,3) {3};
        \node (B4) at (-1,4) {6};
        \node (B5) at (-1,5) {7};
        \node (B6) at (-1,6) {9};
        \node (C1) at (1,1) {3};
        \node (C2) at (1,2) {3};
        \node (C3) at (1,3) {4};
        \node (C4) at (1,4) {5};
        \node (C5) at (1,5) {8};
        \node (C6) at (1,6) {8};
        \draw[-] (A1) -- (A2);
        \draw[-] (A2) -- (A3);
        \draw[-] (A3) -- (A4);
        \draw[-] (A4) -- (A5);
        \draw[-] (A5) -- (A6);
        \draw[-] (B1) -- (B2);
        \draw[-] (B2) -- (B3);
        \draw[-] (B3) -- (B4);
        \draw[-] (B4) -- (B5);
        \draw[-] (B5) -- (B6);
        \draw[-] (C1) -- (C2);
        \draw[-] (C2) -- (C3);
        \draw[-] (C3) -- (C4);
        \draw[-] (C4) -- (C5);
        \draw[-] (C5) -- (C6);
        \draw[-] (A1) -- (B1);
        \draw[-] (A1) -- (C1);
        \draw[-] (A2) -- (B2);
        \draw[-] (A2) -- (C2);
        \draw[-] (A3) -- (B3);
        \draw[-] (A3) -- (C3);
        \draw[-] (A4) -- (B4);
        \draw[-] (A4) -- (C4);
        \draw[-] (A5) -- (B5);
        \draw[-] (A5) -- (C5);
        \draw[-] (A6) -- (B6);
        \draw[-] (A6) -- (C6);
        \draw[-to] (1.5, 3) --node[above]{$\tau_4$} (2,3);
        \draw[color = blue] (C4) circle (7pt);
        \draw[color = red] (C3) circle (7pt);
        \end{scope}
        \begin{scope}[shift={(14,0)}]
        \node (A1) at (0,0) {1};
        \node (A2) at (0,1) {1};
        \node (A3) at (0,2) {2};
        \node (A4) at (0,3) {2};
        \node (A5) at (0,4) {3};
        \node (A6) at (0,5) {6};
        \node (B1) at (-1,1) {2};
        \node (B2) at (-1,2) {2};
        \node (B3) at (-1,3) {3};
        \node (B4) at (-1,4) {6};
        \node (B5) at (-1,5) {7};
        \node (B6) at (-1,6) {9};
        \node (C1) at (1,1) {3};
        \node (C2) at (1,2) {3};
        \node (C3) at (1,3) {4};
        \node (C4) at (1,4) {5};
        \node (C5) at (1,5) {8};
        \node (C6) at (1,6) {8};
        \draw[-] (A1) -- (A2);
        \draw[-] (A2) -- (A3);
        \draw[-] (A3) -- (A4);
        \draw[-] (A4) -- (A5);
        \draw[-] (A5) -- (A6);
        \draw[-] (B1) -- (B2);
        \draw[-] (B2) -- (B3);
        \draw[-] (B3) -- (B4);
        \draw[-] (B4) -- (B5);
        \draw[-] (B5) -- (B6);
        \draw[-] (C1) -- (C2);
        \draw[-] (C2) -- (C3);
        \draw[-] (C3) -- (C4);
        \draw[-] (C4) -- (C5);
        \draw[-] (C5) -- (C6);
        \draw[-] (A1) -- (B1);
        \draw[-] (A1) -- (C1);
        \draw[-] (A2) -- (B2);
        \draw[-] (A2) -- (C2);
        \draw[-] (A3) -- (B3);
        \draw[-] (A3) -- (C3);
        \draw[-] (A4) -- (B4);
        \draw[-] (A4) -- (C4);
        \draw[-] (A5) -- (B5);
        \draw[-] (A5) -- (C5);
        \draw[-] (A6) -- (B6);
        \draw[-] (A6) -- (C6);
        \draw[-to] (1.5, 3) --node[above]{$\tau_5$} (2,3);
        \draw[color = blue] (A6) circle (7pt);
        \draw[color = blue] (B4) circle (7pt);
        \draw[color = red] (C4) circle (7pt);
        \end{scope}
        \begin{scope}[shift={(3.5,-7)}]
        \node (A1) at (0,0) {1};
        \node (A2) at (0,1) {1};
        \node (A3) at (0,2) {2};
        \node (A4) at (0,3) {2};
        \node (A5) at (0,4) {3};
        \node (A6) at (0,5) {5};
        \node (B1) at (-1,1) {2};
        \node (B2) at (-1,2) {2};
        \node (B3) at (-1,3) {3};
        \node (B4) at (-1,4) {5};
        \node (B5) at (-1,5) {7};
        \node (B6) at (-1,6) {9};
        \node (C1) at (1,1) {3};
        \node (C2) at (1,2) {3};
        \node (C3) at (1,3) {4};
        \node (C4) at (1,4) {6};
        \node (C5) at (1,5) {8};
        \node (C6) at (1,6) {8};
        \draw[-] (A1) -- (A2);
        \draw[-] (A2) -- (A3);
        \draw[-] (A3) -- (A4);
        \draw[-] (A4) -- (A5);
        \draw[-] (A5) -- (A6);
        \draw[-] (B1) -- (B2);
        \draw[-] (B2) -- (B3);
        \draw[-] (B3) -- (B4);
        \draw[-] (B4) -- (B5);
        \draw[-] (B5) -- (B6);
        \draw[-] (C1) -- (C2);
        \draw[-] (C2) -- (C3);
        \draw[-] (C3) -- (C4);
        \draw[-] (C4) -- (C5);
        \draw[-] (C5) -- (C6);
        \draw[-] (A1) -- (B1);
        \draw[-] (A1) -- (C1);
        \draw[-] (A2) -- (B2);
        \draw[-] (A2) -- (C2);
        \draw[-] (A3) -- (B3);
        \draw[-] (A3) -- (C3);
        \draw[-] (A4) -- (B4);
        \draw[-] (A4) -- (C4);
        \draw[-] (A5) -- (B5);
        \draw[-] (A5) -- (C5);
        \draw[-] (A6) -- (B6);
        \draw[-] (A6) -- (C6);
        \draw[-to] (1.5, 3) --node[above]{$\tau_6$} (2,3);
        \draw[-to] (-2, 3) --node[above]{$\tau_5$} (-1.5,3);
        \draw[color = blue] (B5) circle (7pt);
        \draw[color = red] (C4) circle (7pt);
        \end{scope}
        \begin{scope}[shift={(7,-7)}]
        \node (A1) at (0,0) {1};
        \node (A2) at (0,1) {1};
        \node (A3) at (0,2) {2};
        \node (A4) at (0,3) {2};
        \node (A5) at (0,4) {3};
        \node (A6) at (0,5) {5};
        \node (B1) at (-1,1) {2};
        \node (B2) at (-1,2) {2};
        \node (B3) at (-1,3) {3};
        \node (B4) at (-1,4) {5};
        \node (B5) at (-1,5) {6};
        \node (B6) at (-1,6) {9};
        \node (C1) at (1,1) {3};
        \node (C2) at (1,2) {3};
        \node (C3) at (1,3) {4};
        \node (C4) at (1,4) {7};
        \node (C5) at (1,5) {8};
        \node (C6) at (1,6) {8};
        \draw[-] (A1) -- (A2);
        \draw[-] (A2) -- (A3);
        \draw[-] (A3) -- (A4);
        \draw[-] (A4) -- (A5);
        \draw[-] (A5) -- (A6);
        \draw[-] (B1) -- (B2);
        \draw[-] (B2) -- (B3);
        \draw[-] (B3) -- (B4);
        \draw[-] (B4) -- (B5);
        \draw[-] (B5) -- (B6);
        \draw[-] (C1) -- (C2);
        \draw[-] (C2) -- (C3);
        \draw[-] (C3) -- (C4);
        \draw[-] (C4) -- (C5);
        \draw[-] (C5) -- (C6);
        \draw[-] (A1) -- (B1);
        \draw[-] (A1) -- (C1);
        \draw[-] (A2) -- (B2);
        \draw[-] (A2) -- (C2);
        \draw[-] (A3) -- (B3);
        \draw[-] (A3) -- (C3);
        \draw[-] (A4) -- (B4);
        \draw[-] (A4) -- (C4);
        \draw[-] (A5) -- (B5);
        \draw[-] (A5) -- (C5);
        \draw[-] (A6) -- (B6);
        \draw[-] (A6) -- (C6);
        \draw[-to] (1.5, 3) --node[above]{$\tau_7$} (2,3);
        \draw[color = red] (C4) circle (7pt);
        \draw[color = blue] (C5) circle (7pt);
        \draw[color = blue] (C6) circle (7pt);
        \end{scope}
        \begin{scope}[shift={(10.5,-7)}]
        \node (A1) at (0,0) {1};
        \node (A2) at (0,1) {1};
        \node (A3) at (0,2) {2};
        \node (A4) at (0,3) {2};
        \node (A5) at (0,4) {3};
        \node (A6) at (0,5) {5};
        \node (B1) at (-1,1) {2};
        \node (B2) at (-1,2) {2};
        \node (B3) at (-1,3) {3};
        \node (B4) at (-1,4) {5};
        \node (B5) at (-1,5) {6};
        \node (B6) at (-1,6) {9};
        \node (C1) at (1,1) {3};
        \node (C2) at (1,2) {3};
        \node (C3) at (1,3) {4};
        \node (C4) at (1,4) {7};
        \node (C5) at (1,5) {7};
        \node (C6) at (1,6) {8};
        \draw[-] (A1) -- (A2);
        \draw[-] (A2) -- (A3);
        \draw[-] (A3) -- (A4);
        \draw[-] (A4) -- (A5);
        \draw[-] (A5) -- (A6);
        \draw[-] (B1) -- (B2);
        \draw[-] (B2) -- (B3);
        \draw[-] (B3) -- (B4);
        \draw[-] (B4) -- (B5);
        \draw[-] (B5) -- (B6);
        \draw[-] (C1) -- (C2);
        \draw[-] (C2) -- (C3);
        \draw[-] (C3) -- (C4);
        \draw[-] (C4) -- (C5);
        \draw[-] (C5) -- (C6);
        \draw[-] (A1) -- (B1);
        \draw[-] (A1) -- (C1);
        \draw[-] (A2) -- (B2);
        \draw[-] (A2) -- (C2);
        \draw[-] (A3) -- (B3);
        \draw[-] (A3) -- (C3);
        \draw[-] (A4) -- (B4);
        \draw[-] (A4) -- (C4);
        \draw[-] (A5) -- (B5);
        \draw[-] (A5) -- (C5);
        \draw[-] (A6) -- (B6);
        \draw[-] (A6) -- (C6);
        \draw[-to] (1.5, 3) --node[above]{$\tau_8$} (2,3);
        \draw[color = blue] (B6) circle (7pt);
        \draw[color = red] (C6) circle (7pt);
        \end{scope}
        \begin{scope}[shift={(14,-7)}]
        \node (A1) at (0,0) {1};
        \node (A2) at (0,1) {1};
        \node (A3) at (0,2) {2};
        \node (A4) at (0,3) {2};
        \node (A5) at (0,4) {3};
        \node (A6) at (0,5) {5};
        \node (B1) at (-1,1) {2};
        \node (B2) at (-1,2) {2};
        \node (B3) at (-1,3) {3};
        \node (B4) at (-1,4) {5};
        \node (B5) at (-1,5) {6};
        \node (B6) at (-1,6) {8};
        \node (C1) at (1,1) {3};
        \node (C2) at (1,2) {3};
        \node (C3) at (1,3) {4};
        \node (C4) at (1,4) {7};
        \node (C5) at (1,5) {7};
        \node (C6) at (1,6) {9};
        \draw[-] (A1) -- (A2);
        \draw[-] (A2) -- (A3);
        \draw[-] (A3) -- (A4);
        \draw[-] (A4) -- (A5);
        \draw[-] (A5) -- (A6);
        \draw[-] (B1) -- (B2);
        \draw[-] (B2) -- (B3);
        \draw[-] (B3) -- (B4);
        \draw[-] (B4) -- (B5);
        \draw[-] (B5) -- (B6);
        \draw[-] (C1) -- (C2);
        \draw[-] (C2) -- (C3);
        \draw[-] (C3) -- (C4);
        \draw[-] (C4) -- (C5);
        \draw[-] (C5) -- (C6);
        \draw[-] (A1) -- (B1);
        \draw[-] (A1) -- (C1);
        \draw[-] (A2) -- (B2);
        \draw[-] (A2) -- (C2);
        \draw[-] (A3) -- (B3);
        \draw[-] (A3) -- (C3);
        \draw[-] (A4) -- (B4);
        \draw[-] (A4) -- (C4);
        \draw[-] (A5) -- (B5);
        \draw[-] (A5) -- (C5);
        \draw[-] (A6) -- (B6);
        \draw[-] (A6) -- (C6);
        \end{scope}
    \end{tikzpicture}
    \caption{The steps of $\Pro$ when the Bender--Knuth involutions are applied to $f$ of Figure~\ref{fig:V69Example}. At each step in the computation of $\Pro$, the raisable $k$ labels are circled in red and the lowerable $k+1$ labels are circled in blue.}
    \label{fig:Examplepromotion}
\end{figure}
\subsection{Rowmotion}\label{sec:rowmotion}

We now review rowmotion on the order polytope of $P$ and consequently on bounded $P$-partitions. We first define the order polytope of $P$, following the description given in~\cite{stanley1986two}. Let $\hat{P}$ denote the poset obtained from $P$ by adjoining a new mininmal element $\hat{0}$ and a new maximal element $\hat{1}. $
\begin{definition}[{\cite{stanley1986two}}]
	For a poset $P$ the \emph{order polytope} of $P$ is \[\O(P)=\{f:\hat{P}\to [0,1]|\text{ if }p\le_{\hat{P}} p' \text{ then } f(p)\le f(p')\text{ and } f(\hat{0})=0, f(\hat{1})=1\}.\] Equivalently $\O(P)$ is the set of order preserving functions from $P$ to $[0,1]$.
\end{definition}

We now define rowmotion on $\O(P)$ and on $P$-partitions, where we follow an amalgamation of the treatments given in~\cite{einstein2021combinatorial} and~\cite{bernstein2024p}. For each $p\in P$, we define the toggle at $p$, denoted by $\tau_p$, to be $\tau_p:\O(P)\to \O(P)$ where for any $f\in \O(P)$ and $p'\in P$ \[\tau_p(f)(p') = \begin{cases} f(p') & p'\neq p\\ \min\{f(r) | p'\lessdot r\}+\max\{f(r) | r\lessdot p'\}-f(p') & p' = p
\end{cases}\]
The following facts about the toggles follow immediately. Firstly the toggles are in fact involutions, as the toggle $\tau_p$ just reflects the value of the coordinate indexed by $p$ in the interval of possible values. Secondly, just as was shown for the combinatorial case in~\cite{cameron1995orbits} $p$ and $p'$ do not share a cover relation if and only if the toggles $\tau_p$ and $\tau_{p'}$ commute. Additionally if for any $\ell\in \Z$ and $f\in \O(P)$ such that $\ell f$ is integer valued then for any $p\in P, \ell\tau_p(f)$ integer valued is as well. As such, for every integer $\ell$, we may talk about the action of the toggles restricted to the elements $f\in \O(P)$ such that $\ell f$ is integer valued. Note that these functions are just the maps from $P\to \{0,1,\dots, \ell\}$ that are order preserving, otherwise known as $\ell$-bounded $P$-partitions. We denote these by $\PP^{\ell}(P)$. 
\begin{definition}\label{def:rowmotion}
    Let $(p_1,p_2, \dots, p_m)$ be a linear extension of $P$. Then rowmotion on $\O(P)$, and consequently on $\PP^\ell(P)$, otherwise known as $\emph{piecewise-linear rowmotion}$ is defined as $\row = \tau_{p_1}\circ \tau_{p_2}\circ \cdots \circ \tau_{p_m}$.
\end{definition}

For our purposes the primary relation between $P$-strict promotion and rowmotion that we will use is the following result.
\begin{prop}[{\cite[Corollary 2.26]{bernstein2024p}}]\label{prop:corconj}
    Let $P$ be a graded poset of rank $n$. Then $\mathcal{L}_{P\times [\ell]}(R^q)$ under $\Pro$ is in equivariant bijection with $\PP^\ell(P\times[q-(n+1)])$ under $\row$.
\end{prop}

We note that this equivariant bijection passes through a map we will see later that is called \text{TogPro}. Additionally the equivariant bijection between TogPro and $\row$ does not depend on any linear extension of $P$ and is thus invariant under any automorphism.

As a consequence of the above proposition, by proving Theorem~\ref{thm:Main} for all $\ell$ and $q$ we will have shown that the order of $\row$ on the rational points of $\O(\V\times [q-2])$ has order dividing $2q$, so the order of $\row$ on $\O(\V\times [q-2])$ has order dividing $2q$ by a density argument. 
\subsection{Kreweras Words and Promotion}
We now discuss Kreweras words, which were originally considered by Kreweras in~\cite{kreweras1965classe} as a variant on the three candidate generalization of the ballot problem. Our immediate goal is to generalize Kreweras words and their associated promotion action, as considered by Hopkins and Rubey in~\cite{hopkins2022promotion}. The purpose of these will be to help understand promotion on $\mathcal{L}_{\V\times [\ell]}(R^q)$.
\begin{definition}[{\cite{hopkins2022promotion}}]\label{def:KrewerasWord}
    A \emph{Kreweras word} of length $3n$ is a word in letters $A,B,C$ with equally many $A's$, $B'$s, and $C's$ for which every prefix has at least as many $A's$ as $B's$ and also at least as many $A's$ as $C's$.
\end{definition}

Additionally, these words have an action upon them called \emph{promotion}, which is defined as follows.
\begin{definition}[\cite{hopkins2022promotion}]\label{def:PromotionOfKrewerasWords}
    Let $w=(w_1,w_2,\ldots,w_{3n})$ be a Kreweras word of length $3n$. The \emph{promotion} of~$w$, denoted $\Pro(w)$, is obtained from $w$ as follows. Let $\iota(w)$ be the smallest index $\iota \geq 1$ for which the prefix $(w_1,w_2,\ldots,w_{\iota})$ has either the same number of $A$'s as $B$'s or the same number of $A$'s as $C$'s. Then
\[ \Pro(w)  =  (w_2,w_3,\ldots,w_{\iota(w)-1}, A, w_{\iota(w)+1},  w_{\iota(w)+2}, \ldots, w_{3n}, w_{\iota(w)}).\]

\end{definition}

\begin{figure}[h]
    \centering
   \begin{tikzpicture}[scale=0.75]
\def \x {0.7};
\node at (-1*\x,-.55*\x) {$w=$};
\node at (1*\x,-0.55*\x) {A};
\node at (2*\x,-0.55*\x) {C};
\node at (3*\x,-0.55*\x) {A};
\node at (4*\x,-0.55*\x) {B};
\node at (5*\x,-0.55*\x) {B};
\node at (6*\x,-0.55*\x) {A};
\node at (7*\x,-0.55*\x) {A};
\node at (8*\x,-0.55*\x) {B};
\node at (9*\x,-0.55*\x) {C};
\node at (10*\x,-0.55*\x) {C};
\node at (11*\x,-0.55*\x) {A};
\node at (12*\x,-0.55*\x) {C};
\node at (13*\x,-0.55*\x) {B};
\node at (14*\x,-0.55*\x) {A};
\node at (15*\x,-0.55*\x) {B};
\node at (16*\x,-0.55*\x) {C};
\node at (17*\x,-0.55*\x) {C};
\node at (18*\x,-0.55*\x) {B};

\node at (1*\x,-0.55*\x-1) {A};
\node at (18*\x,-0.55*\x-1) {C};
\node at (2*\x,-0.55*\x-1) {A};
\node at (3*\x,-0.55*\x-1) {B};
\node at (4*\x,-0.55*\x-1) {B};
\node at (5*\x,-0.55*\x-1) {A};
\node at (6*\x,-0.55*\x-1) {A};
\node at (7*\x,-0.55*\x-1) {B};
\node at (8*\x,-0.55*\x-1) {C};
\node at (9*\x,-0.55*\x-1) {C};
\node at (10*\x,-0.55*\x-1) {A};
\node at (11*\x,-0.55*\x-1) {C};
\node at (12*\x,-0.55*\x-1) {B};
\node at (13*\x,-0.55*\x-1) {A};
\node at (14*\x,-0.55*\x-1) {B};
\node at (15*\x,-0.55*\x-1) {C};
\node at (16*\x,-0.55*\x-1) {C};
\node at (17*\x,-0.55*\x-1) {B};

\node at (-1*\x,-.55*\x-1) {$\Pro(w)=$};

\end{tikzpicture}
    \caption{The Kreweras word $w$ whose associated linear extension is given in Figure~\ref{fig:Pro} and $\Pro(w)$.}
    \label{fig:krew}
\end{figure}

It is easy to verify that $\Pro(w)$ is also a Kreweras word, and that promotion is an invertible action on the set of Kreweras words. Linear extensions of $\V\times [n]$ correspond to Kreweras words of length $3n$ as follows. 

If $l$ is a linear extension of $\V\times [n]$ and $l^{-1}(i)=(p,k)$ then $w_i =p$. As noted in \cite{hopkins2022promotion}, this is the same as just forgetting the second coordinate when viewing a linear extension as a word in the letters $A,B,C$. Importantly, as Hopkins and Rubey show in the following Proposition, the promotion actions on Kreweras words of length $3n$ and linear extensions of $\V\times [n]$ are the same.
\begin{prop}[{\cite[Proposition 2.2]{hopkins2022promotion}}]\label{prop:KrewerasWordPromotionIntertwined}
The above map of forgetting the second coordinate is a bijection from linear extensions of~$\V\times[n]$ to Kreweras words of length $3n$, and under this bijection, promotion of linear extensions corresponds to promotion of Kreweras words.
\end{prop}

An additional perspective on these words and how promotion acts is via what is called the \emph{Kreweras bump diagram}, described in \cite{hopkins2022promotion}. To properly state the definition, we include the relevant definitions from \cite{hopkins2022promotion} below.

\begin{definition}[{\cite[Definition \ 3.2]{hopkins2022promotion}}]
  An \emph{arc} is a pair $(i,j)$ of positive integers with $i < j$.  A \emph{crossing} is a set $\{(i,j),(k,\ell)\}$ of two arcs such that $i\leq k < j < \ell$.
\end{definition}

\begin{definition}[{\cite[Definition\ 3.3]{hopkins2022promotion}}]
  Let $\mathcal{A}$ be a collection of arcs. For a set of positive integers  $S$, we say that $\mathcal{A}$ is a \emph{noncrossing matching of $S$} if 
  \begin{itemize}
      \item for every $(i,j)\in \mathcal{A}$ we have $i,j \in S$
      \item every $i\in S$ belongs to a unique arc in~$\mathcal{A}$
      \item no two arcs in $\mathcal{A}$ form a crossing.
  \end{itemize}
  The set of \emph{openers of $\mathcal{A}$} is $\{i\colon (i,j)\in \mathcal{A}\}$ and the set the set of \emph{closers of~$\mathcal{A}$} is $\{j\colon (i,j)\in \mathcal{A}\}$.
\end{definition}
\begin{definition}[{\cite[Definition 3.4]{hopkins2022promotion}}]\label{def:KrewerasBumpDiagram}
Let $w$ be a Kreweras word of length $3n$. Let $\varepsilon \in \{\textrm{$B$},\textrm{$C$}\}$, where $-\varepsilon$ denotes the other element of $\{\textrm{$B$},\textrm{$C$}\}$. We use~$\M^\varepsilon_w$ to denote the noncrossing matching of $\{i\in [3n]\colon w_i\neq -\varepsilon\}$ whose set of openers is~$\{i\in [3n]\colon w_i=\textrm{$A$}\}$ and whose set of closers is~$\{i\in [3n]\colon w_i=\varepsilon\}$.

The \emph{Kreweras bump diagram} $\D_w$ of $w$ is obtained by placing the numbers $1,\ldots,3n$ in order on a line, and drawing a semicircle above the line connecting $i$ and $j$ for each arc~$(i, j)\in \M^{B}_w \cup \M^{C}_w$.  The arc is solid {\bf b}lue if~$(i,j)\in\M^{B}_w$ and dashed {\bf c}rimson (i.e.,~red) if~$(i,j) \in \M^{C}_w$. The arcs are drawn in such a fashion that only pairs of arcs which form a crossing intersect and any two arcs intersect at most once.
\end{definition}

In the proof of the order of promotion on linear extensions of $\V\times [n]$~\cite[Theorem \ 1.2]{hopkins2022promotion}, the Kreweras bump diagram plays a central role. By considering a local rule at the crossings of arcs in the diagram, called the \emph{rules of the road}~\cite[Definition 3.6]{hopkins2022promotion}, Hopkins and Rubey construct a permutation of $3n$, denoted by $\sigma_w$, called the \emph{trip permutation}~\cite[Definition 3.6]{hopkins2022promotion} of $w$. First they showed that $\sigma_w$ together with a sequence of $B$'s and $C$'s coming from $w$ and $\sigma_w$, called $\varepsilon_w$, can uniquely recover $w$. They then show that $\Pro$ corresponds to a rotation of order $3n$ on $\sigma_w$ and a rotation of order $6n$ on $\varepsilon_w$, implying the order of $\Pro$ on $\V\times [n]$ divides $6n$. 

For our purposes, we introduce a generalization the Kreweras bump diagram. It will suffice to decompose our generalizations of Kreweras words by the corresponding arc structure of the generalized diagrams. Once we have these generalizations, we will relate and describe $P$-strict promotion on $\mathcal{L}_{\V\times[\ell]}(R^q)$ in terms of promotion of Kreweras words without any analogues of the rules of the road or trip permutations.

\begin{figure}[htb]
    \centering
\begin{tikzpicture}[scale=0.85]
\def \x {0.7};
\node at (-1*\x,0) {$\D_w=$};
\node (1) at (1*\x,0) {1};
\node (2) at (2*\x,0) {2};
\node (3) at (3*\x,0) {3};
\node (4) at (4*\x,0) {4};
\node (5) at (5*\x,0) {5};
\node (6) at (6*\x,0) {6};
\node (7) at (7*\x,0) {7};
\node (8) at (8*\x,0) {8};
\node (9) at (9*\x,0) {9};
\node (10) at (10*\x,0) {10};
\node (11) at (11*\x,0) {11};
\node (12) at (12*\x,0) {12};
\node (13) at (13*\x,0) {13};
\node (14) at (14*\x,0) {14};
\node (15) at (15*\x,0) {15};
\node (16) at (16*\x,0) {16};
\node (17) at (17*\x,0) {17};
\node (18) at (18*\x,0) {18};
\node at (1*\x,-0.55*\x) {A};
\node at (2*\x,-0.55*\x) {C};
\node at (3*\x,-0.55*\x) {A};
\node at (4*\x,-0.55*\x) {B};
\node at (5*\x,-0.55*\x) {B};
\node at (6*\x,-0.55*\x) {A};
\node at (7*\x,-0.55*\x) {A};
\node at (8*\x,-0.55*\x) {B};
\node at (9*\x,-0.55*\x) {C};
\node at (10*\x,-0.55*\x) {C};
\node at (11*\x,-0.55*\x) {A};
\node at (12*\x,-0.55*\x) {C};
\node at (13*\x,-0.55*\x) {B};
\node at (14*\x,-0.55*\x) {A};
\node at (15*\x,-0.55*\x) {B};
\node at (16*\x,-0.55*\x) {C};
\node at (17*\x,-0.55*\x) {C};
\node at (18*\x,-0.55*\x) {B};
\draw [ultra thick,blue] (1) to [out=90, in=90] (5);
\draw [ultra thick,blue] (3) to [out=90, in=90] (4);
\draw [ultra thick,blue] (6) to [out=90, in=90] (18);
\draw [ultra thick,blue] (7) to [out=90, in=90] (8);
\draw [ultra thick,blue] (11) to [out=90, in=90] (13);
\draw [ultra thick,blue] (14) to [out=90, in=90] (15);
\draw [ultra thick,red,dashed] (1) to [out=90, in=90] (2);
\draw [ultra thick,red,dashed] (3) to [out=90, in=90] (17);
\draw [ultra thick,red,dashed] (6) to [out=90, in=90] (10);
\draw [ultra thick,red,dashed] (7) to [out=90, in=90] (9);
\draw [ultra thick,red,dashed] (11) to [out=90, in=90] (12);
\draw [ultra thick,red,dashed] (14) to [out=90, in=90] (16);
\end{tikzpicture}
    \label{fig:ExampleKrewerasBumpDiagram}
    \caption{The Kreweras bump diagram of the word $w=ACABBAABCCACBABCCB$}
\end{figure}
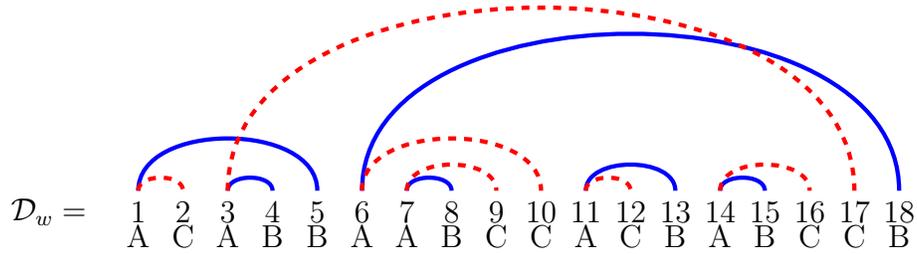

\section{Proof of Main Theorem}\label{sec:Proof}
We now provide the proof of Theorem~\ref{thm:Main}. In a very broad sense, the idea will be to use $P$-strict labelings as a semistandard analogue of linear extensions of $\V\times [n]$ and show that the question of $P$-strict promotion can be reduced to the case of promotion on linear extensions $\V\times [n]$. For added readability, the argument is broken into subsections of similarly related ideas within the overall proof. 
\subsection{Partial Multi Kreweras Words}
To begin, we define the previously alluded to generalization of the Kreweras word. These objects will be our combinatorial model for $P$-strict promotion on $\mathcal{L}_{\V\times [\ell]}(R^q)$. 
\begin{definition}\label{def:KrewWordGen}
	An \emph{$(\ell,q)$-partial multi Kreweras word} is a sequence $w=w_1w_2\dots w_q$ of $q$, potentially empty, multisets of $\{A,B,C\}$ subject to the following conditions. For each $i$ neither the number of $B's$ nor the number of $C's$ in $w_1w_2\dots w_i$ exceeds the number of $A's$ in $w_1w_2\dots w_{i-1}$. Additionally there are $\ell$ of each of $A,B$, and $C$. 
 
We call $w_i$ the $i$th \emph{block} of $w$.
\end{definition}
The collection of $(\ell,q)$-partial multi Kreweras words is in bijection correspondence with $\mathcal{L}_{\V\times[\ell]}(R^q)$ via the map $W$, where $W$ takes a word $w$ to a $\V$-strict labeling as follows. For each $p\in \V$ the number of instances of $p$ in $w_i$ is the number of labels in $F_p$, the fiber above $p$, that are equal to $i$.

When writing one of these words, we will always place the $A's$ in $w_i$ after the $B's$ and/or $C's$. 
Unless otherwise specified we will ignore the order of the $B's$ and $C's$. Additionally, if $w_i = \emptyset$ we write $\emptyset$ in the $i$th position. 

\begin{figure}[htb]\label{fig:LQPMKWofF}
    \centering
    \begin{tikzpicture}
        \node (F) at (0,0) {\begin{tabular}{c c c c c c c c c} A & CA & BBAA & BCCA & C & BA & B & CC &B\\
        1 & 2 & 3 & 4 & 5 & 6 & 7 & 8 & 9
        \end{tabular}};
    \end{tikzpicture}
    \caption{The associated (6,9)-partial multi Kreweras word associated to $f$ of Figure~\ref{fig:V69Example} with the index of $w_i$ written below.}
\end{figure}

We define the actions of the Bender--Knuth involutions, and thus promotion, on these words as follows. For $1\le k \le q-1$,  define $\tau_k(w) := W^{-1}\circ \tau_k\circ W(w)$ and $\Pro(w) = \tau_{q-1}\tau_{q-2}\dots \tau_1(w)$. 

At the level of the word $w$, $\tau_k$ swaps some $A's$, $B's$, and $C's$ between $w_k$ and $w_{k+1}$. It is always possible to swap an $A$ in $w_{k+1}$ to $w_k$ and it is always possible to swap a $B$ or $C$ in $w_k$ to $w_{k+1}$. There is only one way an $A$ in $w_k$ cannot be swapped to $w_{k+1}$ or a $B$ (or $C$) in $w_{k+1}$ cannot be swapped to $w_k$. This is when $w_k$ contains the $i, i+1, \dots, j$th $A$'s of $w$ and $w_{k+1}$ contains the $s, s+1, \dots, t$th $B's$ (or $C$'s) with $[i,j]\cap [s,t]\neq \emptyset$. The reason is that for each $r \in [i,j]\cap [s,t]$ in $W(w)$, $f(A,r)=k, f(B,r)= k+1$, so neither of these labels are free, as in Definition~\ref{def:BenderKnuth}. 

To describe how $\Pro$ impacts $w$, we introduce a generalization of the Kreweras bump diagram. 

\begin{definition}\label{def:KrewasWordBumpDiagramGen}
    Given an $(\ell,q)$-partial multi Kreweras word $w$, linearly order the $A's$ within each block, where the $A's$ follow the $B's$ and $C's.$ Draw the noncrossing arc diagrams as in Definition~\ref{def:KrewerasBumpDiagram} between the $A's$ and $B's$ and the $A's$ and $C's$ using this linear ordering within each block, where the number of arcs in the diagram between the $A's$ and $B$'s of the form $(i,j)$, with $i < j$ are the number of $B$'s in $w_j$. This is just to say that we have degenerate crossings where there can be multiple arcs whose right endpoints share the same location. We call the resulting diagram the \emph{generalized bump diagram of }$w$ and we denote the generalized bump diagram of $w$ by $\D_w$ following the notation of \cite{hopkins2022promotion}.

    Additionally we call the instances where an $A$ has arcs to a $B$ and $C$ in the same block a \emph{double arc}.
\end{definition}

We denote this linear ordering by subscripting the $A$'s. 

\begin{definition}\label{def:DiagramInducedMatchingPartition}
        Suppose $f\in \mathcal{L}_{\V\times [\ell]}(R^q)$ and $w$ is the associated $(\ell,q)$-partial multi Kreweras word. Further suppose that $w$ has generalized bump diagram $\D_w$. For each $i\in [\ell]$, if $A_i$ is in block $w_{a_i}$ and $A_i$ has arcs to a $B$ and $C$, in blocks $w_{b_{i}}, w_{c_{i}}$ respectively, define $L'_i$ to be the $P$-strict labeling of $\V$ with $L'_i(A) = a_i, L'_i(B) = b_i, L'_i(C)= c_i$. We call the multiset of $\V$-strict labelings obtained from $w$ in this way the \emph{noncrossing layer decomposition of }$f$.
\end{definition}

\begin{figure}[htb]
    \centering
\begin{tikzpicture}[scale=0.85]
\def \x {0.7};
\def \y {1.6};
\node at (-1*\x,0) {$\D_w=$};
\node (1) at (1*\y,0) {1};
\node (2) at (2*\y,0) {2};
\node (3) at (3*\y,0) {3};
\node (4) at (4.25*\y,0) {4};
\node (5) at (5.24*\y,0) {5};
\node (6) at (6.25*\y,0) {6};
\node (7) at (7.25*\y,0) {7};
\node (8) at (8.25*\y,0) {8};
\node (9) at (9.25*\y,0) {9};
\node (A1) at (1*\y,0.55*\x) {A${}_1$};
\node (L2)at (2*\y-.25*\x,0.55*\x) {C};
\node (A2) at (2*\y+.05+.25*\x,0.55*\x-.035) {A${}_2$};
\node (L3)at (3*\y-.5*\x,0.55*\x) {BB};
\node (A3) at (3*\y+.25+.05*\x,0.55*\x-.03) {A${}_3$};
\node (A4) at (3*\y+.75+.05*\x,0.55*\x-.03) {A${}_4$};
\node (L4) at (4.25*\y-.25*\x,0.55*\x) {BCC};
\node (A5) at (4.25*\y+.05+.75*\x,0.55*\x-.03) {A${}_5$};
\node (A6) at (6.25*\y+.25*\x+.05,0.55*\x-.03) {A${}_6$};
\node (L5)at (5.25*\y,0.55*\x) {C};
\node (L6) at (6.25*\y-.25*\x,0.55*\x) {B};
\node (L7) at (7.25*\y,0.55*\x) {B};
\node (L8) at (8.25*\y,0.55*\x) {CC};
\node (L9) at (9.25*\y,0.55*\x) {B};

\draw [ultra thick,blue] (A1) to [out=90, in=90] (L3);
\draw [ultra thick,blue] (A2) to [out=90, in=90] (L3);
\draw [ultra thick,blue] (A3) to [out=90, in=90] (L9);
\draw [ultra thick,blue] (A4) to [out=90, in=90] (L4);
\draw [ultra thick,blue] (A5) to [out=90, in=90] (L6);
\draw [ultra thick,blue] (A6) to [out=90, in=90] (L7);
\draw [ultra thick,red,dashed] (A1) to [out=90, in=90] (L2);
\draw [ultra thick,red,dashed] (A2) to [out=90, in=90] (L8);
\draw [ultra thick,red,dashed] (A3) to [out=90, in=90] (L4);
\draw [ultra thick,red,dashed] (A4) to [out=90, in=90] (L4);
\draw [ultra thick,red,dashed] (A5) to [out=90, in=90] (L5);
\draw [ultra thick,red,dashed] (A6) to [out=90, in=90] (L8);
\begin{scope}[shift = {(0,-2)}]
    \node at (-1.5, 0) {$L'_1$ = };
    \node (A) at (0,0) {1};
    \node (B) at (-1,1) {3};
    \node (C) at (1,1) {2};
    \draw [-] (A) -- (B);
    \draw [-] (A) -- (C);
\end{scope}
\begin{scope}[shift = {(3,-2)}]
    \node at (-1.5, 0) {$L'_2$ = };
    \node (A) at (0,0) {2};
    \node (B) at (-1,1) {3};
    \node (C) at (1,1) {8};
    \draw [-] (A) -- (B);
    \draw [-] (A) -- (C);
\end{scope}
\begin{scope}[shift = {(6,-2)}]
    \node at (-1.5, 0) {$L'_3$ = };
    \node (A) at (0,0) {3};
    \node (B) at (-1,1) {9};
    \node (C) at (1,1) {4};
    \draw [-] (A) -- (B);
    \draw [-] (A) -- (C);
\end{scope}
\begin{scope}[shift = {(9,-2)}]
    \node at (-1.5, 0) {$L'_4$ = };
    \node (A) at (0,0) {3};
    \node (B) at (-1,1) {4};
    \node (C) at (1,1) {4};
    \draw [-] (A) -- (B);
    \draw [-] (A) -- (C);
\end{scope}
\begin{scope}[shift = {(12,-2)}]
    \node at (-1.5, 0) {$L'_5$ = };
    \node (A) at (0,0) {4};
    \node (B) at (-1,1) {6};
    \node (C) at (1,1) {5};
    \draw [-] (A) -- (B);
    \draw [-] (A) -- (C);
\end{scope}
\begin{scope}[shift = {(15,-2)}]
    \node at (-1.5, 0) {$L'_6$ = };
    \node (A) at (0,0) {6};
    \node (B) at (-1,1) {7};
    \node (C) at (1,1) {8};
    \draw [-] (A) -- (B);
    \draw [-] (A) -- (C);
\end{scope}
\end{tikzpicture}
    \caption{The associated Kreweras bump diagram to $w$ and noncrossing layer decomposition of $W(w)$, where $w$ is from Figure~\ref{fig:LQPMKWofF}}
    \label{fig:GenKrewerasArcDiagram}
\end{figure}
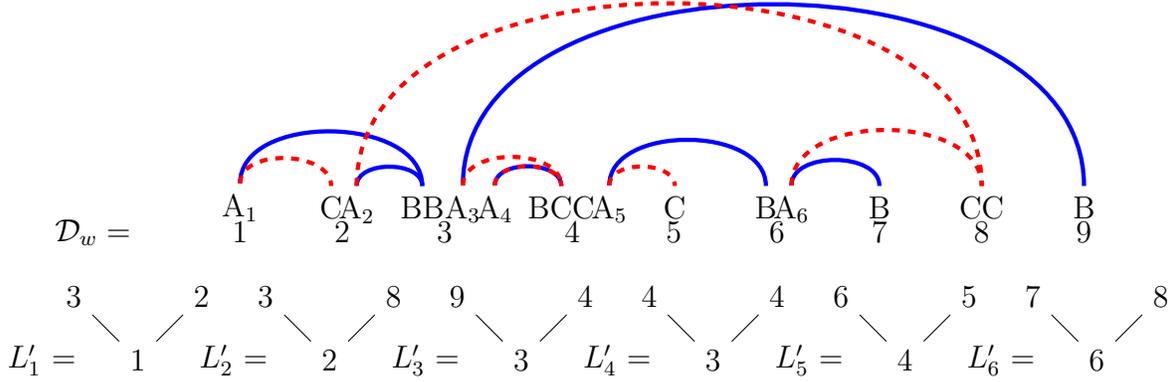
We now state and prove our first result concerning how $\Pro$ impacts $w$.

\begin{prop}\label{prop:ContentRotationViaBumpDiagram}
If $f$ has noncrossing layer decomposition $\{L'_i\}$, then $\Pro(f)$ is the $P$-strict labeling obtained by applying $\Pro$ to each $L'_i$ and then reordering the labels within each fiber.
\end{prop}
\begin{proof}

Let $w = W^{-1}(f)$. If there are no $A$'s labeled 1, then $\Pro$ will reduce each label by 1. This is the same as applying $\Pro$ to each $L'_i$. So we will assume that there is some $A$ with label 1. One thing to note is that in $\tau_{r}\dots\tau_1 f$ when considering the application of $\tau_{r+1}$, the only $A$ labels that can be raised corresponded to $A$'s that before applying any toggles had the label of 1. This is as all other $A$'s corresponded to lowerable labels.

 Let $L'_i(A)= 1$, $L'_i(B)=b_i, L'_i(C)=c_i$, where without loss of generality $b_i\le c_i$, and suppose that $L'_{i+1}(A) > 1$, i.e. this $A$ is largest $A$ in the linear order of $w_1$.  
  Denote this $A$ of $L'_i$ by $A_i$ and suppose that the $B$ of $L'_i$ is the $B$ of the $k$th layer of $f$, where $k$ is minimal among the $B$'s of $w_{b_i}$ that are matched to an $A$ in $w_1$. Under $\tau_{b_i-2}\dots \tau_2\tau_1$ the ${b_i-1}$st block is exactly $i$ $A$'s. This is as every other $A$ label encountered during the applications of $\tau_{b_i-2}\dots \tau_2\tau_1$ was lowerable. Additionally every label corresponding to a $B$ or $C$ that was encountered was lowerable. This claim about the labels of the $B$'s and $C$'s encountered up to this point always being lowerable holds by the following argument. 
  
  The labels that correspond to $B$'s or $C$'s were matched, via the noncrossing matchings, to an $A$ that at the time of checking if the label of the $j$th $B$ or $C$ is lowereable has label at least 2 less than the label of the $j$th $B$ or $C$. This $A$ initially had label at least 1 less than the label of the $j$th $B$ or $C$ but was then lowered. As such there are at least $j$ $A's$ with labels at least 2 less than the label of the $j$th $B$ or $C$. Consequently the label of the $j$th $B$ or $C$ is lowerable. 
  
  Through the application of $\tau_{b_i-2}\dots \tau_2\tau_1$ to $f$ no labels have been fixed, so there are $i$ $A's$ in $w_{b_i-1}$.  Importantly there are exactly $k-1$ $A$'s through the first $b_i-2$ blocks of $\tau_{b_i-2}\dots \tau_2\tau_1(w)$.

  If there were any fewer, then $k$ would not be minimal, and if there were any more, then the $k$th $B$ would not be matched to $A_i$ in the noncrossing matching.
  
Now consider what the application of $\tau_{b_{i}-1}$ will do to $\tau_{b_i-2}\dots \tau_2\tau_1 f$. For convenience let $w' = W^{-1}(\tau_{b_i-2}\dots \tau_2\tau_1 f)$. All the labels of all the $B$'s and $C's$ that were the $j$th $B$ or $C$, for $j < k$, in $f$ correspond to lowerable labels for the same reasons that all previously encountered $B$ and $C$ labels were lowerable. Notice that the $k$th $B$ is in block $w'_{b_i}$ and the $k$th $A$ is in block $w'_{b_i-1}$. So this $A$ will not be a raisable label and this $B$ will not be a lowerable label. If $L'_i(C) = L'_i(B)$ then the label of the associated $C$ will also not be lowerable. 

Note that this argument holds for any $A$ in $w'_{b_i-1}$ with associated noncrossing layer $L'_{i'}$ in $f$ that satisfies $\min(L'_{i'}(B), L'_{i'}(C))=b_i$  Following the same logic, the only $B's$ or $C's$ in $w'_{b_i}$ that were lowerable under the application of $\tau_{b_{i}-1}$ are those that were not matched via the noncrossing matchings to $A$'s that were initially in the first block of $w$.

While then continuing to apply the $\tau$'s, we see that the only time a label corresponding to an $A$ is not raisable when applying $\tau_t$ to $\tau_{t-1}\dots \tau_2\tau_1 f$ is when in the corresponding noncrossing layer $L'_r$ decomposition of $f$ that the $A$ corresponding to the label which is not raisable was matched to a $B$ or $C$ that was in block $w_{t+1}$ and $\min(L'_r(B), L'_r(C))=t+1$. As such all labels that were not associated to a noncrossing layer $L'_s$, with $s \le i$, have just been reduced by 1. 

Since the $A$ to which the $B$ or $C$ was originally matched to in $\mathcal{D}_w$ has corresponding label at least 2 less when applying the first toggle which can change the label, then so were all labels that were associated to a noncrossing layer that were strictly larger than $\min(L'_s(B), L'_s(C)$, by the same logic that was used to show that all the $B$'s and $C's$ that were before the $k$th $B$ corresponded to lowerable labels.  

It also follows immediately that if a label corresponding to a $B$ or $C$ was not lowerable, then in all subsequent toggles the associated label will be raisable. Consequently there will be a $B$ or $C$ in the last block of $W^{-1}\Pro(f)$ for each label of a $B$ or $C$ that was not lowerable.  

In addition note that for any $g\in \mathcal{L}_\V(R^q)$, which is just an increasing labeling of $\V$ with entries in $[q]$, $\Pro(g)$ just reduces each label by $1$ if $g(A) >1$. If $g(A)=1$, there are two cases. When $g(B)=g(C)$, $\Pro(g)$ is obtained by first increasing the label of $A$ to be 1 less than $g(B)$, and then increasing the labels of $B$ and $C$ to $q$. 
Otherwise, without loss of generality assuming that $g(B) < g(C)$, then $\Pro(g)$ is obtained by first increasing the label of $A$ to be 1 less than $g(B)$, then decreasing $g(C)$ by 1, and setting $g(B)=q$. 
Putting this all together, observe that the labels of $\Pro(f)$ were changed exactly as if $\Pro$ had been applied to each labeling in the noncrossing layer decomposition.

\end{proof}

\begin{figure}[htb]
    \centering
\begin{tikzpicture}[scale=0.85]
\def \x {0.7};
\def \y {1.6};
\node at (-1*\x,0) {$\D_w=$};
\node (1) at (1*\y,0) {1};
\node (2) at (2*\y,0) {2};
\node (3) at (3*\y,0) {3};
\node (4) at (4.25*\y,0) {4};
\node (5) at (5.24*\y,0) {5};
\node (6) at (6.25*\y,0) {6};
\node (7) at (7.25*\y,0) {7};
\node (8) at (8.25*\y,0) {8};
\node (9) at (9.25*\y,0) {9};
\node (A1) at (1*\y,0.55*\x) {A${}_1$};
\node (L2)at (2*\y-.25*\x,0.55*\x) {C};
\node (A2) at (2*\y+.05+.25*\x,0.55*\x-.035) {A${}_2$};
\node (L3)at (3*\y-.5*\x,0.55*\x) {BB};
\node (A3) at (3*\y+.25+.05*\x,0.55*\x-.03) {A${}_3$};
\node (A4) at (3*\y+.75+.05*\x,0.55*\x-.03) {A${}_4$};
\node (L4) at (4.25*\y-.25*\x,0.55*\x) {BCC};
\node (A5) at (4.25*\y+.05+.75*\x,0.55*\x-.03) {A${}_5$};
\node (A6) at (6.25*\y+.25*\x+.05,0.55*\x-.03) {A${}_6$};
\node (L5)at (5.25*\y,0.55*\x) {C};
\node (L6) at (6.25*\y-.25*\x,0.55*\x) {B};
\node (L7) at (7.25*\y,0.55*\x) {B};
\node (L8) at (8.25*\y,0.55*\x) {CC};
\node (L9) at (9.25*\y,0.55*\x) {B};

\draw [ultra thick,blue] (A1) to [out=90, in=90] (L3);
\draw [ultra thick,blue] (A2) to [out=90, in=90] (L3);
\draw [ultra thick,blue] (A3) to [out=90, in=90] (L9);
\draw [ultra thick,blue] (A4) to [out=90, in=90] (L4);
\draw [ultra thick,blue] (A5) to [out=90, in=90] (L6);
\draw [ultra thick,blue] (A6) to [out=90, in=90] (L7);
\draw [ultra thick,red,dashed] (A1) to [out=90, in=90] (L2);
\draw [ultra thick,red,dashed] (A2) to [out=90, in=90] (L8);
\draw [ultra thick,red,dashed] (A3) to [out=90, in=90] (L4);
\draw [ultra thick,red,dashed] (A4) to [out=90, in=90] (L4);
\draw [ultra thick,red,dashed] (A5) to [out=90, in=90] (L5);
\draw [ultra thick,red,dashed] (A6) to [out=90, in=90] (L8);
\begin{scope}[shift = {(0,-2)}]
    \node at (-1.5, 0) {$L'_1$ = };
    \node (A) at (0,0) {1};
    \node (B) at (-1,1) {3};
    \node (C) at (1,1) {2};
    \draw [-] (A) -- (B);
    \draw [-] (A) -- (C);
\end{scope}
\begin{scope}[shift = {(3,-2)}]
    \node at (-1.5, 0) {$L'_2$ = };
    \node (A) at (0,0) {2};
    \node (B) at (-1,1) {3};
    \node (C) at (1,1) {8};
    \draw [-] (A) -- (B);
    \draw [-] (A) -- (C);
\end{scope}
\begin{scope}[shift = {(6,-2)}]
    \node at (-1.5, 0) {$L'_3$ = };
    \node (A) at (0,0) {3};
    \node (B) at (-1,1) {9};
    \node (C) at (1,1) {4};
    \draw [-] (A) -- (B);
    \draw [-] (A) -- (C);
\end{scope}
\begin{scope}[shift = {(9,-2)}]
    \node at (-1.5, 0) {$L'_4$ = };
    \node (A) at (0,0) {3};
    \node (B) at (-1,1) {4};
    \node (C) at (1,1) {4};
    \draw [-] (A) -- (B);
    \draw [-] (A) -- (C);
\end{scope}
\begin{scope}[shift = {(12,-2)}]
    \node at (-1.5, 0) {$L'_5$ = };
    \node (A) at (0,0) {4};
    \node (B) at (-1,1) {6};
    \node (C) at (1,1) {5};
    \draw [-] (A) -- (B);
    \draw [-] (A) -- (C);
\end{scope}
\begin{scope}[shift = {(15,-2)}]
    \node at (-1.5, 0) {$L'_6$ = };
    \node (A) at (0,0) {6};
    \node (B) at (-1,1) {7};
    \node (C) at (1,1) {8};
    \draw [-] (A) -- (B);
    \draw [-] (A) -- (C);
\end{scope}
\end{tikzpicture}

\vspace{.2cm}
\begin{tikzpicture}[scale=0.85]
\def \x {0.7};
\def \y {1.6};
\node at (-1*\x,0) {$\D_{\Pro(w)}=$};
\node (1) at (.9*\y,0) {1};
\node (2) at (2*\y,0) {2};
\node (3) at (3.25*\y,0) {3};
\node (4) at (4.25*\y,0) {4};
\node (5) at (5.24*\y,0) {5};
\node (6) at (6.25*\y,0) {6};
\node (7) at (7.25*\y,0) {7};
\node (8) at (8.25*\y,0) {8};
\node (9) at (9.25*\y,0) {9};
\node (A1) at (.8*\y,0.55*\x) {A${}_1$};
\node (L2)at (2*\y-.55*\x,0.55*\x) {BB};
\node (A2) at (.8*\y+.5+.05*\x,0.55*\x) {A${}_2$};
\node (L3)at (3*\y+.2*\x,0.55*\x) {BCC};
\node (A3) at (2*\y+.25+.05*\x,0.55*\x-.03) {A${}_3$};
\node (A4) at (2*\y+.75+.05*\x,0.55*\x-.03) {A${}_4$};
\node (L4) at (4.25*\y,0.55*\x) {C};
\node (A5) at (3.25*\y+.7*\x,0.55*\x-.03) {A${}_5$};
\node (A6) at (5.1*\y+.65*\x+.05,0.55*\x-.03) {A${}_6$};
\node (L5)at (5.1*\y,0.55*\x) {B};
\node (L6) at (6.25*\y,0.55*\x) {B};
\node (L7) at (7.25*\y,0.55*\x) {CC};
\node (L8) at (8.25*\y,0.55*\x) {B};
\node (L9) at (9.25*\y,0.55*\x) {C};

\draw [ultra thick,blue] (A1) to [out=90, in=90] (L2);
\draw [ultra thick,blue] (A2) to [out=90, in=90] (L2);
\draw [ultra thick,blue] (A3) to [out=90, in=90] (L8);
\draw [ultra thick,blue] (A4) to [out=90, in=90] (L3);
\draw [ultra thick,blue] (A5) to [out=90, in=90] (L5);
\draw [ultra thick,blue] (A6) to [out=90, in=90] (L6);
\draw [ultra thick,red,dashed] (A1) to [out=90, in=90] (L9);
\draw [ultra thick,red,dashed] (A2) to [out=90, in=90] (L7);
\draw [ultra thick,red,dashed] (A3) to [out=90, in=90] (L3);
\draw [ultra thick,red,dashed] (A4) to [out=90, in=90] (L3);
\draw [ultra thick,red,dashed] (A5) to [out=90, in=90] (L4);
\draw [ultra thick,red,dashed] (A6) to [out=90, in=90] (L7);
\begin{scope}[shift = {(0,5.1)}]
    \node at (-1.5, 0) {$\Pro(L'_1)$ = };
    \node (A) at (0,0) {1};
    \node (B) at (-1,1) {2};
    \node (C) at (1,1) {9};
    \draw [-] (A) -- (B);
    \draw [-] (A) -- (C);
\end{scope}
\begin{scope}[shift = {(3,5.1)}]
    \node at (-1.5, 0) {$\Pro(L'_2)$ = };
    \node (A) at (0,0) {1};
    \node (B) at (-1,1) {2};
    \node (C) at (1,1) {7};
    \draw [-] (A) -- (B);
    \draw [-] (A) -- (C);
\end{scope}
\begin{scope}[shift = {(6,5.1)}]
    \node at (-1.5, 0) {$\Pro(L'_3)$ = };
    \node (A) at (0,0) {2};
    \node (B) at (-1,1) {8};
    \node (C) at (1,1) {3};
    \draw [-] (A) -- (B);
    \draw [-] (A) -- (C);
\end{scope}
\begin{scope}[shift = {(9,5.1)}]
    \node at (-1.5, 0) {$\Pro(L'_4)$ = };
    \node (A) at (0,0) {2};
    \node (B) at (-1,1) {3};
    \node (C) at (1,1) {3};
    \draw [-] (A) -- (B);
    \draw [-] (A) -- (C);
\end{scope}
\begin{scope}[shift = {(12,5.1)}]
    \node at (-1.5, 0) {$\Pro(L'_5)$ = };
    \node (A) at (0,0) {3};
    \node (B) at (-1,1) {5};
    \node (C) at (1,1) {4};
    \draw [-] (A) -- (B);
    \draw [-] (A) -- (C);
\end{scope}
\begin{scope}[shift = {(15,5.1)}]
    \node at (-1.5, 0) {$\Pro(L'_6)$ = };
    \node (A) at (0,0) {5};
    \node (B) at (-1,1) {6};
    \node (C) at (1,1) {7};
    \draw [-] (A) -- (B);
    \draw [-] (A) -- (C);
\end{scope}
\end{tikzpicture}
    \caption{The associated generalized Kreweras bump diagram to $w$ and noncrossing layer decomposition of $W(w)$, where $w$ is from Figure~\ref{fig:LQPMKWofF}}
    \label{fig:GenKrewerasArcDiagramUnderPro}
\end{figure}
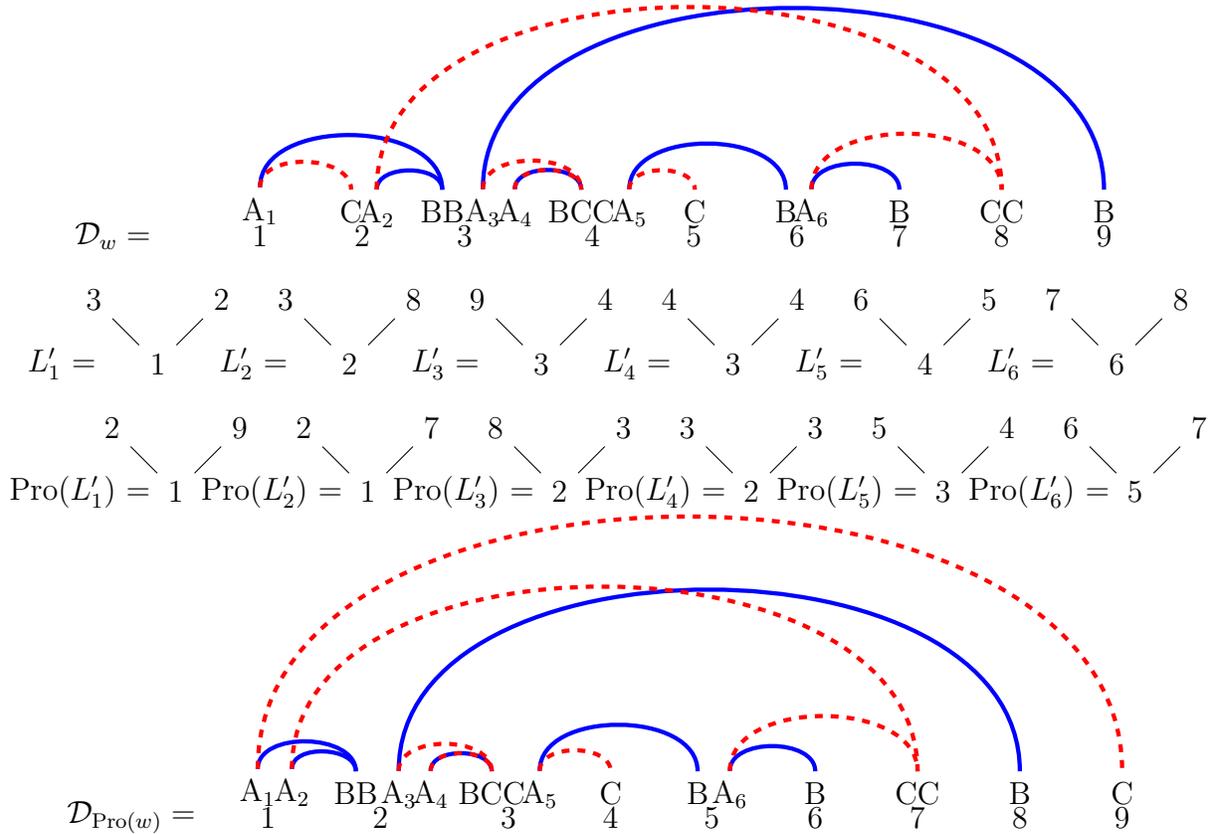

\subsection{Double Arcs}
Next, we try to understand the behavior of arcs under $\Pro$. We will fully describe the behavior of double arcs to reduce to the case where there are no double arcs. 
\begin{lemma}\label{lem:doubleArcRot}
For each $L'_i$ in the noncrossing layer decomposition of $W(w)$ that corresponds to a double arc between $(k,j)$ in $\mathcal{D}_w$ there is a noncrossing layer in the noncrossing layer decomposition of $\Pro(W(w))$ which is $\Pro(L'_i)$, i.e. a double arc between $(k-1,j-1)$ if $k > 1$ and otherwise a double arc between $(j-1,q)$ in $\mathcal{D}_{\Pro(w)}$.
\end{lemma} 
\begin{proof}
     First observe that no $A_s$ for $s\le i$ can have an arc to a $B$ or $C$ in any of the blocks $w_t$ for $t\in [k+1,j-1]$. Such an arc would be part of a crossing in one of the two matchings. There are then two cases to consider, either $k>1$ or $k=1$. 
     
     If $k > 1$ and in $\mathcal{D}_{\Pro(w)}$ there is not a corresponding double arc from $(k-1,j-1)$, then one of the $B$ or $C$ from $L'_i$ is matched to an $A$ in $\mathcal{D}_{\Pro(w)}$ that was in block 1 of $w$. This is as for there not to be a double arc, there would need to be an $A$ in block $s\in [k,j-2]$ in $\Pro(w)$ that was not in block $s+1$ in $w$. But this cannot occur, as it would imply this $A$ in $\mathcal{D}_w$ was matched to a $B$ or $C$ that was in block $s$, which would imply a crossing in the corresponding matching. 

If $k=1$, then for every $i' < i$, $A_{i'}$ is matched to a $B$ and $C$ in respective blocks indexed by $b_{i'},c_{i'}\ge j$ with $\min{b_{i'},c_{i'}}=j'$ due to the noncrossing property of the matchings. When considering $\Pro(w)$ there will be an $A$ in block $j'-1\ge j-1$, by Proposition~\ref{prop:ContentRotationViaBumpDiagram}, for each $i' < i$, and a $B$ and a $C$ in blocks with indices at least $j'$ which are matched to the $A$ corresponding to $A_{i'}$ in $w$. This means that the $A$ that was $A_i$ in $w$ now corresponds to an $A$ in block $j-1$ which must match to a $B$ and $C$ in block $q$. Every $B$ and $C$ in $w$ that was matched to an $A_{i'}$ with $i' < i$, or an $A$ in a block $t,t\ge j$, corresponds to a $B$ or $C$ in a block $s\ge j$ by Proposition~\ref{prop:ContentRotationViaBumpDiagram}.  In addition the associated $A$ must be in a block indexed by $s'\ge j-1$. After rearranging $A's$ within a block we can assume that all of the associated $A's$ in $\Pro(w)$ follow the $A$ that was $A_i$, so they match to all the $B$'s and the $C'$s they collectively were associated to in $w$. As such the $A$ that was $A_i$ in $w$ must match to a $B$ and $C$ in $\Pro(w)_q$, as there are no other $B$'s and $C's$ to match to. So there is a double arc of the form $(j-1,q)$ in $\Pro(w)$ that corresponded to the double arc of the form $(1,j)$ in $w$. 

Note that if there are multiple double arcs of the form $(k,j)$ in $\mathcal{D}_w$, $\Pro$ acts identically on all of them. This is because they are interchangeable at the level of the word. As such, they each correspond to a double arc of the form $(k-1,j-1)$ if $k> 1$ or $(j-1,q)$ if $k=1$. 
\end{proof}

We now show that the number of double arcs is preserved under $\Pro$. The proof provided is more involved than necessary, but provides more understanding of the structure of $P$-strict promotion. Additionally some of the machinery will be essential later. A shorter proof is the following. By Lemma~\ref{lem:doubleArcRot} the number of double arcs of $\Pro(w)$ is at least the number of double arcs of $w$. Since $\Pro$ has finite order, the number of double arcs can never strictly increase.  So the number of double arcs must be constant over an orbit of $\Pro$.

In $\D_w$ with associated $P$-strict labeling $f$, for $A_i$ associated to $L'_i$, if $L'_i(B) \le L'_i(C)$ (or $L'_i(C) \le L'_i(B)$), we say that the arc $(a_i,b_i)$ (or $(a_i,c_i)$) is the \emph{shortest arc associated to $A_i$}.

\begin{lemma}\label{lem:ShortestArcsPreserved}
     If $f(A_i)= a_i > 1$ and the shortest arc in $\D_w$ of $A_i$ is $(a_i, b_i)$ (or $(a_i,c_i)$), then in $\D_{\Pro(w)}$ there is an $A$ in block $a_i-1$ with shortest arc to a $B$ (or $C$) in block $b_i-1$ ($c_i-1$).
\end{lemma}

\begin{proof}
    Without loss of generality, assume $b_i \le c_i$. If the shortest arc associated to $A_i$ does not just shift down by 1 block in each coordinate under $\Pro$, then it must be of the form $(a_i-1, j)$ with $j\ge b_i$ as otherwise $(a_i,b_i)$ wouldn't have been the shortest arc associated to $A_i$. 
    So the $B$ that was originally matched to $A_i$ shifted down 1 block by Proposition~\ref{prop:ContentRotationViaBumpDiagram}. In $\Pro(w)$ there must be an $A$ that follows the $A$ that was $A_i$ in $w$ that did not do so in $w$.  
    This $A$ must have been an $A$ in block 1 whose shortest arc was to a $B$ or $C$ which followed $A_i$ but preceded the $B$ of the shortest arc of $A_i$. But this can't happen, as it would imply that one of the matchings has a crossing. 
\end{proof}
\begin{lemma}\label{lem:NoNewDoubleArcs}
    For a word $w$ associated to a $P$-strict labeling $f$, the number of double arcs of $\D_{\Pro(w)}$ equals the number of double arcs of $\D_w$.
\end{lemma}
\begin{proof}
    By Lemma~\ref{lem:doubleArcRot}, we need only show that no new double arcs are created. If a double arc could be created, then it must be associated to an $A$ that is in the first block of $w$. 
    To see why, if an $A$ in block $i>1$ in $w$ has a double arc in $\Pro(w)$, by Lemma~\ref{lem:ShortestArcsPreserved} the shortest arc in $\D_w$ from this $A$ must be going to a block $w_d$ with both a $B$ and $C$. Specify this $A$ as $A'$ and suppose its shortest arc is to a $B$.
    Since there was not a double arc in $\D_w$ associated to $A'$, then there must be an $A$ that follows $A'$ in the linear order of the $A's$ with shortest arc to a $B$ that is in a block which precedes $w_d$ that was matched to the $C$ in the double arc in $\D_{\Pro(w)}$.
   We can assume this second $A$ is not part of a double arc, since if so we have just relabeled and not created a new double arc. Consequently this second $A$ must be matched to a $B$ in a block which strictly precedes $w_d$. Thus the swapping of the arcs to the $C$'s of these two $A$'s would induce a crossing in the matching between $A's$ and $C's$, as the $C$ to which $A'$ is matched is in a block strictly following $w_d$.

    The $A$ in $w$ which in $\D_{\Pro(w)}$ is part of a new double arc, call it $A_D$, must be in the first block of $w$ as otherwise by Lemma~\ref{lem:ShortestArcsPreserved} $A_D$ would already form a double arc. Additionally, $A_D$ cannot be matched to either of the $B$ or the $C$ with which it will form a double arc. This is because if $A_D$ did, then there would be an $A$ which follows $A_D$ that is matched to the other $B$ or $C$ and does not form a double arc. This $A$ must follow the $B$ or $C$ that is matched in the shortest arc of $A_D$, in which case no double arc would be formed. Then the $B$ and $C$ which will form a double arc must be matched to different $A's$ which follow $A_D$. But this would then force the shortest arc of $A_D$ to cross one of the arcs connecting to these $B$ and $C$, as it must be connected to a $B$ or $C$ which strictly precedes the two. Therefore there can be no new double arcs.
\end{proof}
For a double arc $D$ in $\D_w$, we say the \emph{interior} of $D$ is the set of arcs of $\D_w$ connected to an $A$ with both arcs nested beneath the double arc. Similarly the \emph{exterior} consists of all other arcs. Importantly, no $A$ can have arcs in both the interior and exterior of a double arc, as it would induce a crossing in one of the matchings. We now have everything needed to show that the removal of double arcs does not impact $\Pro$, formalized in the following Proposition.

\begin{prop}\label{prop:IgnoreDoubleArcs}
        Suppose that $w$ is an $(\ell,q)$-partial multi Kreweras word where the associated arc diagram $\D_w$ has a double arc $D$ of the form $(k,j)$. Let $w_D$ be the $(\ell-1,q)$-partial multi Kreweras word obtained from $w$ by deleting $D$. Then the resulting word obtained by deleting the double arc corresponding to $\Pro(D)$ in $\Pro(w)$ is $\Pro(w_D)$.
\end{prop}
\begin{proof}
Note that $W(w)$ and $W(w_D)$ have the same noncrossing layer decomposition aside from the layer corresponding to $D$.  Let $\Int_D(w)$ and $\Ext_D(w)$ denote the $P$-strict labelings corresponding to the labels of the interior and exterior of $D$ respectively. Since there is no overlap between these two collection of arcs, the noncrossing layer decomposition of $W(w)$ is their union together with the layer corresponding to $D$. So by Proposition~\ref{prop:ContentRotationViaBumpDiagram}, $\Pro(w)$ is obtained by applying $\Pro$ to layers corresponding to $D$, $\Int_D(w)$, and $\Ext_D(w)$ and then combining. Following the same reasoning, $\Pro(w_D)$ is obtained by applying $\Pro$ to the layers corresponding to $\Int_D(w)$ and $\Ext_D(w)$ and then combining. The only difference in these layer decompositions is the layer corresponding to $\Pro(L'_D)$, where $L'_D$ is the layer corresponding to $D$. By Lemma~\ref{lem:doubleArcRot} the layer corresponding to $L'_D$ is just $\Pro(L'_D)$, so deleting this layer before or after applying $\Pro$ will make no difference in the corresponding word. 
\end{proof}
\subsection{Standardization and Completing the Proof}
As a consequence of Proposition~\ref{prop:IgnoreDoubleArcs}, as with Lemma~\ref{lem:doubleArcRot} and Lemma~\ref{lem:NoNewDoubleArcs} we not only fully understand how double arcs are impacted under $\Pro$ but also that we can ignore them; we conclude that it suffices to consider the case where there are no double arcs in $\D_w$. Suppose then that $w$ is an $(\ell,q)$-partial multi Kreweras word with no double arcs in $\mathcal{D}_w$. The \emph{standardization} of $w$, $\std(w)$, is the Kreweras word of length $3\ell$ obtained by first linearly ordering the $B$'s and $C$'s of each block of $w$ such that there are no crossings between arcs that terminate in the same block, and then extending the linear orders on the blocks to a linear order of all the letters. The standardization is well defined, as the only such ordering without crossings of arcs that terminate in the same block is the following: within each block the $B$'s and $C$'s are ordered such that the arcs terminating in this block are nesting. Note that the standardization is not an invertible function, see Figure~\ref{fig:stdNoninvert}, but with the information of what the size of each block was, the original word can be recovered uniquely by replacing the labels of the standardization with the multiset of labels of the original word in increasing order.  The final result needed to prove the main theorem is how $\Pro(w)$ impacts $\std(w)$.
\begin{figure}[htb]
    \centering
        \begin{tikzpicture}
        \node (F) at (0,0) {\begin{tabular}{c c c c} $\emptyset$ & AA & CC & BB\\
        1 & 2 & 3 & 4 
        \end{tabular}};
        \node (STD) at (5,0) {\begin{tabular}{c c c c c c } A & A & C & C & B &B\\
        1 & 2 & 3 & 4 & 5 & 6
        \end{tabular}};
        \node (G) at (10,0) {\begin{tabular}{c c c c } A & A & CC & BB\\
        1 & 2 & 3 & 4 
        \end{tabular}};
    \end{tikzpicture}
    \caption{A pair of $(2,4)$ Multi-Partial Kreweras words together with their equal standardization.}
    \label{fig:stdNoninvert}
\end{figure}
\begin{lemma}\label{lem:ProOnStandardization}
    Suppose that $w$ is an $(\ell,q)$-partial multi Kreweras word with $|w_1|$= $k$ and where $\mathcal{D}_w$ has no double arcs. Then $\std(\Pro(w))= \Pro^k(\std(w))$. 
\end{lemma}  
\begin{proof}
   Given a $(\ell,q)$-partial multi Kreweras word $w$ with $|w_1|=k$, consider $\std(w)$. In $\std(w)$, let $A_1,A_2, \dots A_k$ denote the first $k$ $A$'s of $\std(w)$ and $\epsilon_i$ is the $B$ or $C$ that is matched to $A_i$ via the shortest arc. The shortest arc is always well defined as $\mathcal{D}_w$ has no double arcs. Then consider $\Pro(\std(w))$. Observe that $\Pro(\std(w))$ is obtained by shifting all letters that aren't $A_1$ and $\epsilon_1$ forward one space, placing $A_1$ in the space before $\epsilon_1$, and placing $\epsilon_1$ at the end. For $i >1$, $A_1$ precedes $A_i$ and $\epsilon_1$ follows $\epsilon_i$, so in $\Pro(\std(w))$ the $A$ that corresponded to $A_1$ has no arcs which cross any of the arcs between $A_i$ and $\epsilon_i$ for $i >1$. Then following the proof of \cite[Proposition 3.10]{hopkins2022promotion}, which shows that arcs which do not cross the shortest arc of the first $A$ are just shifted down by $1$ in each coordinate under $\Pro$, for the next $k-1$ iterations of $\Pro$ on $\Pro(\std(w))$ the arcs connecting to the $A$ which corresponded to $A_1$ will just shift down by $1$ in each coordinate. By Lemma~\ref{lem:ShortestArcsPreserved} the $A$ corresponding to $A_i$ still has shortest arc to $\epsilon_i$ through the first $i-1$ applications of $\Pro$ on $\std(w)$. Consequently $\Pro^k(\std(w))$ is obtained by shifting each letter which was not an $A_i$ or $\epsilon_i$ forward by $k$ positions, placing an $A$ exactly $k$ positions before the position of each $\epsilon_i$, and the last $k$ letters are $\epsilon_1\epsilon_2\dots \epsilon_k$. Additionally, there is no crossing among arcs connecting to the final $k$ letters since for all $1\le i \le k$ if $i < j$, the $A$ to which $\epsilon_i$ is matched is preceded by the $A$ to which $\epsilon_j$ is matched.

Denote by $A_{1}', A_{2}', \dots , A_{k}'$ the $A's$ in $w_1$ and by $\epsilon'_i$ the $B$ or $C$ in $w$ to which $A_i'$ has its shortest arc. Note that if the letter which follows $\epsilon'_i$ in $\std(w)$, and is not $\epsilon'_{i-1}$, is in the same block as $\epsilon_i'$, then it must be an $A$, since if it is a $B$ or $C$ it must be different than $\epsilon_i'$, as otherwise, because it is not $\epsilon'_{i-1}$, there would be more $B$'s or $C$'s at that point than $A$'s. Similarly it cannot be different due to the lack of double arcs. Consequently we have that in a block the $\epsilon_i'$'s are the terminal sequence of non-$A$ letters.
By Proposition ~\ref{prop:ContentRotationViaBumpDiagram} we know that $\Pro(w)$ is the $P$-strict labeling obtained by shifting each label not associated to $A_i'$ or $\epsilon_i'$ down by 1, having the label associated to $A_i'$ be 1 less than that of $\epsilon_i'$, and having the label of $\epsilon_i$ become $q$. 
This implies that the multiset of the values of the labels has corresponded to cyclically shifting each element down by 1. Additionally, for each $A$, $B$, or $C$ that was not an $A'_i$ or $\epsilon_i'$ there are $k$ fewer preceding letters. Then consider $\std(\Pro(w))$. Observe that the computation for $\Pro(w)$ is the same as deleting $w_1$, replacing each $\epsilon'_i$ with $A'_i$, adding a new artificial block labeled by $q+1$ equal to the multiset of the $\epsilon'_i$'s, and then reducing the label of each block by 1. 
What this corresponds to for $\std(\Pro(w))$ is the same as deleting the first $k$ letters, replacing the $\epsilon_i's$ with $A$'s, adding $k$ letters corresponding to the $\epsilon_i$'s in order at the end, and then shifting the indices down by $k$. Note then that this is the same as $\Pro^k(\std(w))$.
\end{proof}

This is the final tool needed to prove our main result.

\begin{proof}[Proof of Theorem~\ref{thm:Main}]
Suppose that $\D_w$ contains some number of double arcs and consider $\Pro^q(w)$. By Lemmas~\ref{lem:doubleArcRot} and~\ref{lem:NoNewDoubleArcs}, it follows that the double arcs of $\D_{\Pro^q(w)}$ are the same as in $\D_w$, since the endpoints of each double arc were just shifted by $q\mod q$. So by Proposition~\ref{prop:IgnoreDoubleArcs} we can reduce to the case where $w$ has no double arcs.

If $\D_w$ has no double arcs, consider $\Pro^q(w)$. Note that by Proposition~\ref{prop:ContentRotationViaBumpDiagram} the multiset of values for the labels will be the same. Additionally, one can notice that through the $q$ applications, there will be exactly $3\ell$ 1's in the multisets of labels, as the number of instances of each label cyclically rotates. By Lemma~\ref{lem:ProOnStandardization} $\std(\Pro^q(w))=\Pro^{3\ell}(\std(w))$, which by \cite[Theorem 1.2]{hopkins2022promotion} is the reflection of the labels. Then since the standardization is invertible if the multiset of values is known, $\Pro^q(w)$ is just swapping all instances of $B$'s and $C's$. Thus $\Pro^{2q}(w)=w$. 
\end{proof}

\subsection{Periodicity of Piecewise-Linear Rowmotion}\label{sec:Consequences}

In this section, we examine the impact of Theorem~\ref{thm:Main} on $\row$ on $\O(P)$ beyond just implying the finite periodicity. First, we will utilize the equivariance of $\Pro$ acting on $\mathcal{L}_{\V\times [\ell]}(R^q)$ and $\row$ acting on $\PP^\ell(\V\times [q-2])$ to show that $\row^q$ is also just the reflection of the labels. 
We will utilize the fact that this reflection, denoted Flip, is an automorphism of $\V$.

In a more general setting, if $P$ is a graded poset of rank $n$ and $\psi$ is an automorphism of $P$, we have an action of $\psi$ on $P\times [k]$ for every $k$ where $\psi((p,i))=(\psi(p),i)$ for all $(p,i)\in P\times [k]$. This induces an action on $\O(P\times [k])$ by $\psi(f((p,i)))= f((\psi(p),i))$.

Before stating the technical lemma that will be key to the proof that $\row^q$ acts by reflecting the labels on $\O(P)$, we state the intermediate bijection used in the proof of Proposition~\ref{prop:corconj}, known as \emph{toggle-promotion}, defined more generally in~\cite{bernstein2021p}, between $\Pro$ and $\row$. As before we assume $P$ is graded.

\begin{definition}[\cite{bernstein2024p}]
      \emph{Toggle-promotion} on $\PP^{\ell}(P\times[q-n-1])$ is defined as the toggle composition $\text{TogPro} := \cdots \circ\tau_3 \circ \tau_{2}\circ \tau_{1}$, where $\tau_{k}$ denotes the composition of all the $\tau_{(p,k)}$ over all $p \in P$, such that $(p,i)\in P\times [q-n-1]$ and $ i = q-n+\text{rk}(p)-k$.
\end{definition}

\begin{lemma}\label{lem:AutCommute}
        For $P$ a graded poset with $\psi$ an automorphism of $P$, the action of $\psi$ commutes with the equivariant bijection between $\text{TogPro}$ and $\row$ on $\O(P\times [q-(n+1)])$.
\end{lemma}
\begin{proof}
    To begin suppose that $L$ is an arbitrary ordering of $P\times [k]$. Then consider the realization of $\O(P\times [k])$ where the coordinate for $(p,i)$ is $L((p,i))$. Let $L' = L((\psi(p),i)).$ Note that for any $f\in \O(P\times [k])$, $\psi(f)$ is the same as if we instead chose the realization given by $L'$. As TogPro, $\row$, and $\varphi$ are all defined independently of the choice of coordinates, $\psi$ will commute with all of them, as they all commute with a change of coordinates.
\end{proof}
\begin{prop}
    The action of $\row^q$ on $\O(\V\times [q-2])$ is equal to the action of Flip.
\end{prop}
\begin{proof}
    As a consequence of Theorem~\ref{thm:Main} we know that $\text{Flip}\circ\text{TogPro}^q$ is the identity on the rational points of $\O(\V\times [q-2])$ and thus on $\O(\V\times[q-2])$. Then consider $\text{Flip}\circ \row^q= \text{Flip}\circ \varphi\circ \text{TogPro}^q \circ \varphi^{-1}$. By Lemma~\ref{lem:AutCommute}, $\text{Flip}\circ \varphi\circ \text{TogPro}^q \varphi^{-1} = \varphi\circ\text{Flip}\circ \text{TogPro}^q\circ  \varphi^{-1}$ which is the identity. So $\text{Flip} = \row^{-q}=\row^q$ as $\row^q$ is an involution. 
\end{proof}
\bibliographystyle{alphaurl}
\bibliography{bib}
\end{document}